\documentclass[a4,11pt]{article}
\usepackage{graphicx}
\usepackage{epstopdf}
\usepackage{amsmath,amsthm,amsfonts,amssymb,epsfig,graphicx} 
\usepackage{multirow}
\usepackage{rotating}
\usepackage{url} 

\usepackage{algorithm}
\usepackage{algorithmic}

\newcommand{\RR}{{\mathbb{R}}}

\newcommand{\0}{\mbox{\bf 0}}
\newcommand{\1}{\mbox{\bf 1}}
\newcommand{\x}{\mbox{\bf x}}

\renewcommand{\r}{\mbox{\bf r}}

\newcommand{\vb}{\mbox{\bf v}}

\newcommand{\al}{\mbox{\boldmath$\alpha$}}

\newcommand{\mub}{\mbox{\boldmath$\mu$}}
\newcommand{\e}{\mbox{\bf e}}

\makeatletter
\newsavebox\myboxA
\newsavebox\myboxB
\newlength\mylenA
\newcommand*\xoverline[2][0.75]{%
    \sbox{\myboxA}{$\m@th#2$}%
    \setbox\myboxB\null
    \ht\myboxB=\ht\myboxA%
    \dp\myboxB=\dp\myboxA%
    \wd\myboxB=#1\wd\myboxA
    \sbox\myboxB{$\m@th\overline{\copy\myboxB}$}
    \setlength\mylenA{\the\wd\myboxA}
    \addtolength\mylenA{-\the\wd\myboxB}%
    \ifdim\wd\myboxB<\wd\myboxA%
       \rlap{\hskip 0.5\mylenA\usebox\myboxB}{\usebox\myboxA}%
    \else
        \hskip -0.5\mylenA\rlap{\usebox\myboxA}{\hskip 0.5\mylenA\usebox\myboxB}%
    \fi}

\DeclareMathOperator{\supp}{supp} 
\DeclareMathOperator{\argmin}{argmin} 
\DeclareMathOperator{\argmax}{argmax} 
\DeclareMathOperator{\diag}{Diag}

\DeclareMathOperator{\nnz}{nnz} 

\date{} 

\usepackage{xcolor}
\definecolor{brightpink}{rgb}{1.0, 0.0, 0.5}

\newcommand{\revise}[1]{{{\color{black} #1}}}

\newtheorem{theorem}{Theorem}[section]

\newtheorem{lemma}[theorem]{Lemma}

\newtheorem{remark}[theorem]{Remark}
\newtheorem{example}[theorem]{Example}

\begin{document}

\title{Assigning 
 Stationary Distributions to \\  
 Sparse Stochastic Matrices} 

\author{
Nicolas Gillis\thanks{Department of Mathematics and Operational Research, Universit\'e de Mons, 
Rue de Houdain 9, 7000 Mons, Belgium. 
Email: nicolas.gillis@umons.ac.be. NG acknowledges the support
 by the European Union (ERC consolidator, eLinoR, no 101085607).  
 } \and Paul Van Dooren\thanks{UCLouvain, Email: vandooren.p@gmail.com.}
 }

\maketitle

\begin{abstract}
The target stationary distribution problem (TSDP) is the following: given an irreducible stochastic matrix $G$ and a target stationary distribution $\hat \mub$, construct a minimum norm perturbation, $\Delta$, such that  $\hat G = G+\Delta$ is also stochastic and has the prescribed target stationary distribution, $\hat \mub$. 
In this paper, we revisit the TSDP under a constraint on the support of $\Delta$, that is, on the set of non-zero entries of $\Delta$. This is particularly meaningful in practice since one cannot typically modify all entries of $G$. 
We first show how to construct a feasible solution $\hat G$ that has essentially the same support as the matrix $G$. 
Then we show how to compute globally optimal and sparse solutions using the component-wise $\ell_1$ norm and linear optimization. We propose an efficient implementation that relies on a column-generation approach which allows us to solve sparse problems of size up to $10^5 \times 10^5$ in a few minutes.  
We illustrate the proposed algorithms with several numerical experiments.  

\smallskip


\smallskip


\end{abstract}

\section{Introduction} \label{sec:intro}

\revise{
In this paper, we consider discrete-time Markov chains with $n$ nodes. We denote the transition matrix of such a  chain as $G \in \mathbb{R}^{n \times n}_+$ where the entry of $G$ at position $(i,j)$, denoted $G_{i,j}$, is the probability to go to node $j$ when leaving node $i$. The matrix $G$ is stochastic, that is, its entries are nonnegative and the sum of the entries in each row is equal to one. 
The matrix $G$ is irreducible if it is not reducible, that is, if there does not exist two sets of indices, $\mathcal{I}$ and $\mathcal{J}$,  such that $\mathcal{I} \cup \mathcal{J} = \{1,2,\dots,n\}$, $\mathcal{I} \cap \mathcal{J} = \emptyset$ and  
 $G_{i,j} = 0$ for all $i \in \mathcal{I}$ and $j \in \mathcal{J}$. 
 Irreducibility of $G$ implies that there exists a unique positive vector $\mub > 0$ such that $\mub^\top G  = \mub^\top$ and $\mub^\top \1_n = 1$; this is a consequence of the Perron-Frobenius theorem; see~\cite{berman1994nonnegative} for more details. 
 This vector $\mub$ is called the stationary distribution of $G$, and accounts for the ``importance'' of each node, as the $i$th entry of $\mub$ is equal to the probability to be in node $i$ after having spent enough time in the chain.} 

 \revise{
Stochastic matrices $G$ are used to 
model hyperlink networks used in particular for the  PageRank computation~\cite{gleich2015pagerank}, social networks~\cite{wasserman1994social}, or queuing networks~\cite{bolch2006queueing}. In such Markov chains, the stationary distributions contain important information on the nodes in the network, such as their centrality or other types of rankings~\cite{avrachenkov2010quasi}. 
The target stationary distribution then captures some desired properties of the system. For example, consider a road network where each node is a road, and $G_{i,j}$ the probability that a car on
 road $i$ turns into road  $j$, while $G_{i,i}$ can be chosen to model the different travel times along
 the different roads~\cite{crisostomi2011google}. 
 The $i$th entry of the stationary distribution $\mub$  represents the long-range time average with which a typical car will be found to drive on road $i$. In other words, the vector $\mub$ represents the (relative) road congestion; see~\cite{BerkhoutHV23} for more details and other applications.  
}

In this context, the target stationary distribution problem (TSDP) was introduced recently  in~\cite{BerkhoutHV23}, and is defined as follows. We are given 
\begin{itemize} 
    \item $G$, an $n\times n$ irreducible stochastic matrix  with positive stationary distribution $\mub > 0$, that is, 
\begin{equation}  \label{stochastic} G \ge 0,   \quad G\1_n = \1_n, \quad \mub^\top G  = \mub^\top, \quad \mub^\top\1_n=1, 
\end{equation}
where $\1_n$ is the $n$-dimensional vector of all 1's. 

\item $\hat \mub > 0$, a positive target distribution such that $\hat\mub^\top \1_n=1$. 
\end{itemize} 
The TSDP requires to find a minimum norm correction, $\Delta$, such that $\hat G :=G+\Delta$ is still stochastic and has the target $\hat \mub$ as its stationary distribution. The set  of admissible candidate matrices, $\mathcal{D}$,  
is thus described by \revise{three} conditions: 
\begin{equation}   \label{Delta}
\mathcal{D}:= \left\{ \Delta \in \RR^{n\times n} \; \mid \; \hat\mub^\top (G+\Delta)  = \hat\mub^\top, \quad \Delta \1_n=\0_n, \quad G+\Delta \ge 0 \right\}, 
\end{equation} 
where $\0_n$ is the $n$-dimensional vector of all 0's; see~\cite{gregory1992row} for more details on this feasible set.  
In practice, one is interested in reaching the target distribution with minimum effort, that is, in finding minimum norm solutions $\Delta$ yielding a perturbed model $\hat G=G+\Delta$ with the prescribed distribution $\hat\mub$. 
The set $\mathcal{D}$ is convex and the TSDP requires to solve the following convex optimization problem
\begin{equation} \label{eq:TSDP}
\underset{\Delta\in \mathcal{D}}{\min}  \; \|\Delta\|, \tag{TSDP} 
\end{equation}
for a given norm $\| \cdot \|$. 
\revise{For example, in the road network application, one is interested to modify the network at minimum cost  (e.g., add a few road segments, and/or remove/add one-way streets) in order to reduce the congestion.}

 This line of research is quite different from the body
of literature on the sensitivity of the stationary distribution of a stochastic matrix $G$ with respect to a perturbation $\Delta$ that preserves the stochasticity of $G+\Delta$. This is a field of active research and various sensitivity bounds have been proposed in the literature \cite{Abbas} \cite{Caswell} \cite{ChoMeyer}, \cite{Ipsen} \cite{Kirkland} \cite{Meyer} \revise{\cite{mitrophanov2005sensitivity}} \cite{Seneta}. 
Such bounds are of interest in a wide range of application areas, such as mathematical physics, climate modeling, Bayesian statistics, and bio-informatics. But the aim of the present paper is to assign a target distribution, which is an inverse problem rather than a sensitivity problem. 

\revise{
To the best of our knowledge, there are only a few works that consider the TSDP or variants. An early work that considered a similar problem is~\cite{billera2001geometric}. 
They prove that the famous Metropolis-Hastings (M-H) algorithm~\cite{metropolis1953equation, hastings1970monte} (see~\eqref{eq:MHalso} below for the formula) finds the optimal solution of the TSDP where 
    the norm minimized is the weighted component-wise $\ell_1$ norm\footnote{
    In~\cite{billera2001geometric}, authors do not take into account the diagonal entries in the objective. However, by the stochastic constraint and the fact that the M-H solution satisfies $\hat G(i,j) \leq G(i,j)$ for all $i \neq j$, this is equivalent, since $\sum_{j \neq i} |G(i,j)-\hat G(i,j)| = \frac{1}{2} \|G(i,:)-\hat G(i,:)\|_1$ for all $i$. 
    }, 
    $\sum_{i=1}^n \hat \mu_i \|G(i,:)-\hat G(i,:)\|_1$,
and with the additional constraint that $\hat G$ is $\hat \mu$-reversible, that is, 
\[
\hat \mu_i \hat G(i,j) = 
\hat \mu_j \hat G(j,i), \quad \text{ for all  } i,j. 
\]
The M-H solution can be computed in $O(\nnz(G))$ operations, as follows: 
\begin{equation} \label{eq:MHalso}
\hat G(i,j) \; = \;  \min \left( G(i,j) , 
\frac{\hat \mu_j}{\hat \mu_i} G(j,i) 
\right) \quad \text{ for all  } i \neq j,  
\end{equation}
while $\hat G(i,i) = 1-\sum_{j \neq i} \hat G(i,j)$ for all $i$. 
Note that $\hat G(i,j) = 0$ whenever $G(i,j) = 0$ or $G(j,i) = 0$. This implies that the support of $G$ is contained in $\supp(I) \cup \left( \supp(G) \cap \supp(G^\top) \right)$. If the support of $G$ is not symmetric, that is, $\supp(G) \neq \supp(G^\top)$, then the M-H algorithm might fail to generate an irreducible chain $\hat G$; see, e.g., the discussion in~\cite{diaconis1995we}. For example, take the irreducible cycle  
\[
G = \left[
\begin{array}{ccc}
    0  &  1 & 0 \\
     0  &  0 & 1 \\
      1  &  0 & 0 \\
\end{array}
\right], 
\]
with $\mub = [1/3,1/3,1/3]$, and let $\hat \mub = [1/2,1/4,1/4]$. The solution to the M-H algorithm is the identity matrix, that is,  $\hat G = I_n$, which is not irreducible, while the optimal solution is given by 
\[
\Delta^* = \left[ 
\begin{array}{ccc}
    0.5  &  -0.5 & 0 \\
     0  &  0 & 0\\
      0  &  0 & 0 \\
\end{array}
\right], \text{ so that }  G + \Delta^* = \left[ 
\begin{array}{ccc}
    0.5  &  0.5 & 0 \\
     0  &  0 & 1 \\
      1  &  0 & 0 \\
\end{array}
\right]. 
\] 
Note that this solution will coincide with our closed-form solution proposed  in section~\ref{sec:convex}.  
}

Solving the TSDP for large-scale problems directly using commercial solvers might be intractable for large $n$, with $O(n^2)$ variables and constraints; see the discussion in section~\ref{sec:linopt}. 
This motivated the introduction of a heuristic approach to find an approximate minimizer using a set of rank-1 corrections~\cite{BerkhoutHV23}. 
An algorithm was given in~\cite{BerkhoutHV23} for constructing a minimum norm rank-1 solution when such a solution exists. However, for sparse matrices, there might not exist feasible rank-1 perturbations. Moreover, rank-one perturbations are typically dense, which is not desirable: in practice, it is typically not possible to modify most of the links in a network, but rather one would like to modify only a few. For example, in a road network~\cite{crisostomi2011google}, it is not possible to add links between distant roads, while one would like to reduce the congestion by modifying as few existing links as possible.  

In this paper, we therefore consider the TSDP with a support constraint while trying to promote sparse solutions. 
\revise{The support of a matrix  $\Delta \in \mathbb{R}^{n \times n}$ 
is defined as the set containing the index pairs where $\Delta$ is non-zero, that is, 
\[
\supp(\Delta) = \{ (i,j) \ | \ \Delta_{i,j} \neq 0 \}.  
\]  
The support constraint requires that the support of the perturbation, $\Delta$, 
is contained in a given set $\Omega \subseteq \{ (i,j) \ | \ 1 \leq i,j \leq n\}$.
}  
Finally, we focus in this paper on the following TSDP problem:  
\begin{equation} \label{eq:suppTSDP}
\underset{\Delta\in \mathcal{D}}{\min}  
\; \|\Delta\|_1 \quad \text{ such that } \quad \supp(\Delta) \subset \Omega, 
\end{equation}
where 
\begin{itemize}
    \item $\|\Delta\|_1 = \sum_{i,j} |\Delta_{i,j}|$ is the component-wise $\ell_1$ norm which promotes sparsity as it is the convex envelope (that is, the tightest convex lower bound) of the $\ell_0$ norm on the $\ell_\infty$ ball; see, e.g.,~\cite{bach2012optimization} and the references therein. In fact, we will prove in Theorem~\ref{lem:sparsitysol} that if \eqref{eq:suppTSDP} is feasible, there is an optimal solution with less than $\min(|\Omega|,\nnz(G)+2n)$ non-zero entries, \revise{where  $|\Omega|$ is the cardinality of $\Omega$ and $\nnz(G)$ is the number of non-zero entries in $G$.} 

\item $\Omega$ is a given \revise{set of index pairs} that correspond the links that can be modified in the Markov chain.  
\end{itemize} 
For example, if one can only modify existing links, we have $\Omega = \supp(G)$, which is referred to as a non-structural perturbation~\cite{breen2015stationary}.  

The above problem is related to the feasible set $\{ G  \ | \ G \geq 0, \; G \1_n = \1_n, \; \supp(G) \subset \Omega\}$ which was  studied in~\cite{breen2015stationary}, \cite{gregory1992row}, and where authors tried to describe all possible stationary vectors of irreducible matrices in this set, for a given $\Omega$. 
\revise{Note that researchers have tackled other optimization problems over Markov chains; for example, how to choose the non-zero entries of $G$ such that the modulus of the second largest eigenvalue is minimized. This is the so-called problem of finding the fastest mixing Markov chain on the graph induced by $G$~\cite{boyd2004fastest}. 
}   


\paragraph{Contribution and outline of the paper} 

In section~\ref{sec:convex}, we first show that the TSDP with support constraint~\eqref{eq:suppTSDP} always admits a feasible solution for $\Omega = \supp(G+I)$, and provide a feasible solution that is an optimal diagonal scaling of $I-G$.  
 We then discuss and analyze this feasible solution, 
 in particular its distance to optimality, when it is rank-one \revise{and optimal}, and how the ordering of $\hat \mub$ affects the solution. 
In section~\ref{sec:linopt}, we propose an efficient linear optimization formulation of~\eqref{eq:suppTSDP}, with $2n$ equalities and $O(|\Omega|)$ variables, that allows us to solve~\eqref{eq:suppTSDP} in $O(n^3)$ operations, in the worst case. When $\Omega$ is sparse, the run time is empirically observed to be comparable to  the time of constructing a sparse  $2n \times O(\nnz(G))$ matrix with 2 nonzeros per column, where $\nnz(G)$ is the number of non-zero entries in $G$, allowing us to solve sparse problems of size $10^4 \times 10^4$ within seconds, and of size $10^5 \times 10^5$ within minutes.   
When $\Omega$ is dense and $G$ is sparse, we propose a column-generation approach that allows us to solve problems of similar size  in a comparable computational time, by initializing the solution with   the feasible case $\Omega = \supp(G+I)$. 
We report numerical experiments in section~\ref{sec:Num}. 



\paragraph{Notation} We use uppercase for matrices, boldface lowercase for vectors and normal lowercase for scalars. 
We use $\0_n$ and $\1_n$ to denote the $n$-vectors of all 0's and 1's, respectively, and $\e_i$ to denote the $i$th basis vector in $\RR^n$. The $n\times n$ matrix of all 1's is denoted by $\1_{n,n}$. 
We drop the index $n$ when the dimension is clear from the context. 
\revise{
The entry of matrix $M$ at position $(i,j)$ is denoted  by $M_{i,j}$, and the $i$th entry of the vector $\vb$ by $v_i$. The $i$th row (resp.\ $j$th column) of $M$ is denoted by $M_{i,:}$ (resp.\ by $ M_{:,j}$).  
}
By $\supp(M)$, we mean the set of \revise{index pairs} for which the matrix $M$ is nonzero, and by $\nnz(M) = |\supp(M)|$ the number of nonzeros in $M$ where $| \cdot |$ denotes the cardinality of a set. 
By $\diag(\alpha_1,\ldots,\alpha_n)$, we mean the $n\times n$ diagonal matrix with the parameters $\alpha_i$ on its main diagonal, 
and  $( \cdot )^\top$ denotes the transpose. 
The inequalities $>$ and $\ge$ applied to vectors and matrices are elementwise inequalities. 
 We say that an $n$-by-$n$ matrix $M$ is stochastic if $M \geq 0$ and $M \1_n = \1_n$.

\section{TSDP with $\Omega = \supp(G+I_n)$} \label{sec:convex} 

In this section, we analyze a special case where we allow the non-zero entries of $G$ and entries on its diagonal to be modified, that is, $\Omega = \supp(G+I_n)$. 
We first show that the TSDP~\eqref{eq:suppTSDP} is  always feasible in this case (section~\ref{sec:closedform}). 
In fact, we provide an explicit closed-form solution. 
We then discuss properties of this solution: optimality (section~\ref{sec:optimality}), sparsity and rank-1 case (section~\ref{sec:rank1}), and how the ordering of $\mub$ compared to that of $\hat \mub$ affects the solution (section~\ref{sec:reorder}).

\subsection{Closed-form feasible solutions for $\Omega = \supp(G+I_n)$} \label{sec:closedform}

Let us consider the solutions to the TSDP of the form 
\[ 
\Delta(\al) = D(\al) (I_n-G), 
\]
where $D(\al):= \diag(\alpha_1,\ldots,\alpha_n)$ and $\0_n \le \al \le \1_n$ so that $\supp(\Delta) \subseteq \supp(G+I_n)$. This leads to the perturbed matrices 
\begin{equation}  \label{family}  
G(\al) := G + \Delta(\al) = (I_n- D(\al))G + \revise{D(\al)}.  
\end{equation}

\revise{The $i$th row of $G(\al)$ is a convex combination of the $i$th row of $G$ with weight $(1-\alpha_i)$ and  the $i$th row of the identity matrix with weight $\alpha_i$. This means that the larger $\alpha_i$, the more important is the self-loop of node $i$, and hence a larger increase in the corresponding entry of the stationary vector is expected. 
This is a standard construction in Markov chains. 
As we will see below (Theorem~\ref{th:IminusGsolution}), this construction allows us to obtain any stationary distribution.} 

The following lemma formalizes the fact that $G(\al)$ is stochastic, and irreducible if $\al < \1_n$. 
\begin{lemma} \label{lem:Ga}
Let $G$ be an irreducible stochastic matrix.
Then the family of matrices 
\[
{\cal G}_{\al}:=\{G(\al) \ | \ \0_n \le \al \le \1_n \} 
\] is a closed convex set of stochastic matrices, and the subset ${\cal G}_{\al<\1_n}:=\{G(\al) \ | \ \0_n \le \al < \1_n\}$ 
\revise{is the set of} the irreducible stochastic matrices \revise{within} ${\cal G}_{\al}$. 
\end{lemma}
\begin{proof}
Each row of the matrix $G(\al)$ is a convex combination of that of $G$ and $I$, and $G(\al)$ is a closed convex set of stochastic matrices. 
When constraining the parameters to $\0_n \le \al < \1_n$, the matrices $G(\al)$ and $G$ (which is actually equal to $G(\0_n)$) have the same off-diagonal zero pattern, and hence $G(\al)$ is irreducible if and only if $G$ is irreducible (see, e.g., \cite[Chap.~8]{BoydV04}). Finally, in order to show that ${\cal G}_{\al<\1_n}$ contains {\em all} irreducible matrices of ${\cal G}_{\al}$, it suffices to see that if any \revise{$\alpha_i=1$} then $\e_i^\top G(\al)=\e_i^\top$, which implies that $G(\al)$ is reducible.
\end{proof}

The solutions $\Delta(\al)$ lead to a set of stochastic matrices, ${\cal G}_{\al}$. 
It remains to choose the free parameters $\al$  to assign $\hat \mub$ as the stationary distribution of an irreducible stochastic matrix $G(\al)$ in ${\cal G}_{\al<\1_n}$, that is, choose $\al$ such that $\hat \mub^\top G(\al)  = \hat \mub^\top$.

The following theorem shows that the set ${\cal G}_{\al<\1_n}$ always contains feasible solutions for arbitrary $G$ and $\hat\mub$, and provides the minimum norm solution among these feasible solutions.


\begin{theorem} \label{th:IminusGsolution} 
\revise{Let $G$ be an irreducible stochastic matrix with stationary distribution $\mub>0$}, 
and let $\hat \mub>0$, $\hat\mub^\top\1_n=1$, be a
given target stationary distribution. 
Then the set of parameters $\al$ so that the 
matrices $G(\al) \in {\cal G}_{\al<\1_n}$ have the target $\hat \mub$
as stationary distribution, that is, $G(\al)^\top \hat \mub = \hat \mub$, is given by 
$$ 
\al(c)= \1_n - c (\mub./\hat\mub), \quad 0 < c \le 1/\|\mub./\hat\mub\|_\infty, 
$$
where $\mub./\hat\mub$ denotes the element-wise division\footnote{This is the MATLAB notation.} of the vectors $\mub$ and $\hat\mub$.
The parameter vector $\al^* = \al(c^*)$ of minimum norm 
is obtained for $c^*=1/\|\mub./\hat\mub\|_\infty$,  
and it also minimizes the norm of the perturbation matrix $\Delta(\al^*) = D(\al^*)(I_n-G)$, for any \revise{monotone} norm $\| \cdot \|$\revise{, that is, any norm satisfying the property that $|A| \geq |B| \Rightarrow \|A\| \geq \|B\|$ for any matrix $A$ and $B$.}  
\end{theorem}
\begin{proof}
If the vector $\hat \mub$ is the stationary distribution of $G(\al)$, it must satisfy $\hat \mub^\top(I_n-G(\al))=0$.
Since $G(\al)= (I_n-D(\al))G+D(\al) I_n$ for some $\al$, $\hat\mub$ must satisfy
$$ \hat \mub^\top(I_n-D(\al))(I_n-G)=0.$$
But since $G$ is irreducible, the rank of $(I_n-G)$ is $n-1$ and \revise{any} left null vector must be proportional to $\mub^\top$. Therefore, all candidate solutions for $\al$ must satisfy, for some non-zero scalar $c$~:
$$  \hat \mub^\top(I_n-D(\al)) =  c\mub^\top, \quad  \0_n \le \al < \1_n,
$$ 
which can be rewritten as a system of linear equations and inequalities~:
$$   \hat\mu_i(1-\alpha_i) = c \mu_i,  \quad 0 \le \alpha_i < 1, \quad i=1,\ldots,n, 
$$
in the unknown variables $c$ and $\alpha_i, \; i=1,...,n$.
Using the $\circ$-notation for the elementwise product of vectors, this is equivalent to 
$$    \hat\mub\circ (\1_n -\al)  = c \mub, \quad \0_n<(\1_n-\al) \le \1_n.
$$
Since $\mub$ and $\hat\mub$ are positive, and the vector $(\1_n-\al)$ is constrained to be positive, $c$ must be positive as well.
The upper bound $(\1_n-\al)\leq \1_n$ implies that $c$ must be smaller than  or equal to $c^*:=1/\|\mub./\hat\mub\|_\infty$. 
The set of solutions is then the line segment
$$
\al(c) = \1_n -c \mub./\hat\mub, \quad 0< c \le c^*:=1/\|\mub./\hat\mub\|_\infty, 
$$
and the parameter vector $\al(c)$  of minimum norm 
is obtained for $c=c^*$
since all entries of $\al(c)$ decrease linearly with increasing $c$ in the interval $0 < c \le c^*$. The perturbation matrix
$\Delta(\al) = D(\al)(I_n-G)$ is linear in $\al$, and \revise{the absolute values of all entries} decrease with increasing $c$, and hence \revise{any monotone norm} also decreases with increasing $c$, reaching its minimum for $c=c^*$. 
\end{proof}

The following example illustrates the method.

\begin{example} \label{ex:example2}
Consider the stochastic matrix $G$ of a ring network, \revise{in which each node connects to two other nodes, forming a single continuous pathway through each node,}  with stationary distribution~$\mub$:
$$ G=\left[\begin{array}{cccc}  1/2 & 1/4 & 0 & 1/4 \\ 1/4 & 1/2 & 1/4  & 0 \\
 0 & 1/4 & 1/2 & 1/4 \\ 1/4 & 0 & 1/4  & 1/2 \end{array}\right], 
\quad \mub^\top=\left[1/4,1/4,1/4,1/4\right].
$$
Let the target distribution be the vector $\hat\mub^\top=\left[\frac18,\frac18,\frac14,\frac12\right]$. We have $\mub^\top./\hat \mub^\top = \left[2,2,1,\frac12\right]$, with optimal parameter $c^*=\frac12$, so that $\al^* = \al(c^*)=\left[0, 0, \frac12,\frac34\right]^\top$, leading to 
$$ \Delta(\al^*) = \left[\begin{array}{cccc} 
0 & 0 & 0 & 0 \\  
0 & 0 & 0 & 0 \\ 
0 & -1/8 & 1/4 & -1/8 \\
-3/16 & 0 & -3/16 & 3/8 \end{array}\right]  \; 
 \mathrm{and} \; \ 
 G(\al^*)= \left[\begin{array}{cccc} 
 1/2       &     1/4      &      0 &             1/4  \\   
       1/4          &  1/2    &        1/4   &         0    \\    
       0          &    1/8       &     5/8   &         1/4  \\    
       1/16       &    0          &    1/8   &        13/16 \end{array}\right].
$$ 
\revise{
Since $G$ is symmetric, the M-H algorithm, see \eqref{eq:MHalso},  will generate a feasible solution, namely 
$$ \Delta = \left[\begin{array}{cccc} 
0 & 0 & 0 & 0 \\  
0 & 0 & 0 & 0 \\ 
      0        &    -1/8      &      1/8     &       0      \\  
      -3/16      &     0     &        -1/8    &       5/16 
\end{array}
\right]  
\; 
 \mathrm{and} 
 \; \ 
 \hat G
 = \left[\begin{array}{cccc} 1/2 & 1/4 & 0 & 1/4 \\ 1/4 & 1/2 & 1/4  & 0  \\
0 & 1/8 & 3/4 & 1/8 \\ 1/16 & 0 & 1/16 & 7/8 \end{array}\right], 
$$ 
with $\|\Delta\|_1 = 7/8 < \|\Delta(\al^*)\|_1 = 10/8$. 
Note that an optimal solution with error 6/8, computed using linear optimization (see section~\ref{sec:linopt}), is 
$$ \Delta = \left[\begin{array}{cccc} 
0 & 0 & 0 & 0 \\  
0 & 0 & 0 & 0 \\ 
       0       &      -1/8  &          0         &     1/8 \\    
      -3/16      &     0         &    -1/16    &       1/4 
\end{array}
\right]  
\; 
 \mathrm{and} 
 \; \ 
 \hat G
 = \left[\begin{array}{cccc} 
1/2     &     1/4   &        0       &       1/4   \\   
       1/4    &        1/2    &        1/4   &         0   \\     
       0      &        1/8     &       1/2    &        3/8  \\    
       1/16   &        0       &       3/16   &        3/4 
\end{array}\right].
$$      
} 
\end{example}

In the following sections, we discuss several properties of the solutions $G(\al)$. Before doing so, let us introduce some additional notation. 
For given stationary distributions
$\mub$ and $\hat\mub$, we define the ratio $\r:=\mub./\hat\mub$ whose entries belong to the interval $0 < r_*\le r_i \le r^*$, where
\begin{equation} \label{eq:r*} r_*:=\min_i \mu_i/\hat\mu_i >0, \quad \text{ and } \quad r^*:=\max_i \mu_i/\hat\mu_i >0.
\end{equation} 
Note that the maximal admissible value for $c$ is $c^* = 1/r^*$. The elements of $\al(c^*)$ are bounded by 
$$  0 \le   \alpha_i = 1- r_i/r^* \le  1-r_*/r^*=\frac{r^*-r_*}{r^*}, 
$$ 
and
\begin{equation} \label{eq:boundDal}
|\Delta(\al)| = D(\al)|I-G|\le (1- r_*/r^* )|I-G| =\frac{(r^*-r_*)}{r^*}|I-G|.
\end{equation}
Hence $\|\al(c^*)\|$ and $\|\Delta(\al^*)\|$ \revise{for mononote norms $\|.\|$ are expected to} decrease with a decreasing gap $r^*-r_*$, and are equal to zero \revise{if and only if $\mub=\hat\mub$, that is,  $r^*=r_*$}.

\subsection{How good is $\Delta(\al^*)$ ?}   \label{sec:optimality}

In this subsection we compare the optimal solution in $\mathcal{G}_{\al}$ with the lower bound for the optimal 
solution in the larger set $\mathcal{D}$ given in \eqref{Delta}. The following lemma gives such a lower bound.
\begin{lemma} \label{lem:bound}
\revise{Any $\Delta$ such that $\hat \mub^\top (G+\Delta) = \hat \mub^\top$, in particular any feasible solution $\Delta \in \mathcal{D}$ of the TSDP,
satisfies} 
$$ 
\revise{
\| \Delta \| \ge 
\frac{\| \hat\mub^\top(I-G)\|}{\|\hat\mub\|} =  
\frac{\| (\hat\mub-\mub)^\top(I-G)\|}{\|\hat\mub\|},
} 
$$ 
for any induced norm\footnote{\revise{Given a vector norm $\|\cdot\|$, the matrix induced norm  is given by $\|A\| = \sup_{\mathbf z \neq 0} \frac{\|A \mathbf z\|}{\|\mathbf z\|}$.}}, and this also holds for the component-wise $\ell_1$ norm. 
\end{lemma}
\begin{proof} 
The inequality $\| x^\top \|\cdot \| A \|\ge \|x^\top A \|$ for the vector-matrix product $x^\top A$ holds for any induced norm, and also holds for the component-wise $\ell_1$ norm.  
Applying this to $\hat\mub^\top \Delta = \hat \mub^\top(I-G)$ then yields the required inequality for 
$\| \Delta \|$. The equality then follows trivially from $\mub^\top(I-G)=0$. 
\end{proof}
For the perturbation $\Delta(\al) = D(\al)(I_n - G)$,  
the bound 
$\| \Delta(\al^*) \| \le (1-r_*)/r^*\| I_n - G \| $ was obtained earlier; see~\eqref{eq:boundDal}. 
Putting this bound together with Lemma~\ref{lem:bound}, we obtain the following interval
\begin{equation} \label{eq:boundoptDelta}
\| (\hat\mub-\mub)^\top(I_n-G)\| /\|\hat\mub^\top\| 
\; \le \;  \| \Delta(\al^*) \| \; \le \; (1-r_*)/r^* \| I_n - G \|, 
\end{equation} 
for the optimal $\Delta(\al^*)$ in $\mathcal{G}_{\al}$.
This interval goes to zero when $\mathbf{d}:=\hat \mub - \mub$ goes to zero, which indicates that the optimal solution constrained to $\mathcal{G}_{\al}$ is approaching a globally optimal solution. This will be illustrated on numerical examples in section~\ref{sec:Num}. 
In the next section, we 
focus on the case of a rank-1 perturbation.

\subsection{Support of $\Delta(\al^*)$ 
and optimal rank-1 solutions}    \label{sec:rank1}

Using Theorem~\ref{th:IminusGsolution}, it is straightforward to describe the support of the feasible solutions of the TSDP of the form $\Delta(\al) = D(\al)(I_n - G)$. 

\begin{lemma} \label{lem:sparsityDal}
\revise{Let $G$ be an irreducible stochastic matrix with stationary distribution $\mub>0$}, and let us use the same notation as in Theorem~\ref{th:IminusGsolution}. 
For all $\al = \al(c)$ such that $c < c^*$, we have 
$\supp(\Delta(\al)) = \supp(G+I_n)$. 
For $\al^* = \al(c^*)$, \revise{the rows of $\Delta(\al)$ with index in $\left\{ i \ | \ \frac{\mu_i}{\hat \mu_i} = r^* = \frac{1}{c^*} \right\}$ 
are all-zeros, the other rows 
have the same support as that of $\supp(G+I_n)$ (as when $c < c^*$). 
In other terms, for $\al^* = \al(c^*)$,}   
\[
\supp\big( \revise{\Delta(\al^*)_{i,:}}\big)
= \left\{ 
\begin{array}{cc}
 \supp( (G+I_n)_{i,:})   & \text{for } \frac{\mu_i}{\hat \mu_i} < r^* = \frac{1}{c^*},   \\
   \emptyset  & \text{for } \frac{\mu_i}{\hat \mu_i} = r^* = \frac{1}{c^*}. 
\end{array}
\right.
\]
\end{lemma}
\begin{proof} 
 We have   $\Delta(\al)_{i,:} = \alpha_i (\revise{\e_i^\top} - G_{i,:})$. Since $G$ is irreducible $G_{i,:} \neq \e_i^\top$, and hence the supports of the $i$th rows of $\Delta(\al)$ and $I-G$ coincide unless $\alpha_i = 0$. 
 For $\al = \al(c)$, this  happens only when $c = c^*$ and $\frac{\mu_i}{\hat \mu_i} = r^* = \frac{1}{c^*}$. 
\end{proof}



Lemma~\ref{lem:sparsityDal} implies that the  solution $\Delta(\al^*)$ will have the same support as $G+I_n$, except for the rows that correspond to the maximum entries of $\mub./\hat \mub$ which will be equal to zero.  

\paragraph{Rank-1 case} 
Lemma~\ref{lem:sparsityDal} also implies that  $\Delta(\al^*)$  has rank-1 
if the vector $\mub./\hat\mub$ has only two different values, namely
$r_*$ and $r^*$, and if, moreover, $r^*$ is repeated $n-1$ times, in which case $\Delta(\al^*)$ has a single non-zero row with index $j$ such that
\begin{equation} \label{eq:mujmui} \mu_j=r_*\hat\mu_j \quad \text{ and } \quad \mu_i=r^*\hat\mu_i \text{ for all } i\neq j. 
\end{equation}
Since $\mub^\top\1_n=\hat\mub^\top\1_n=1$, $ r^* + \hat\mu_j(r_*-r^*)=1$, we have $\hat \mu_j=(r^*-1)/(r^*-r_*)<1$.
Finally, we obtain
\begin{equation} \label{eq:formrank1}
r_*<1<r^*, \quad \alpha_j=1-\frac{r_*}{r^*}, \quad \alpha_i=0 \text{ for } i\neq j, \quad \Delta =\alpha_j\e_j\e_j^\top(I_n-G) .
\end{equation}
Note that $\e_j\e_j^\top(I_n-G)$ is the matrix whose $j$th row is \revise{$(I_n-G)_{j,:}$} and all the other rows are equal to zero.

\paragraph{What $\hat \mub$ can we choose to make $\Delta(\al^*)$ rank-1?}  We can take 
$\hat \mub 
= \frac{\mub + \lambda \e_j}{(\mub + \lambda \e_j)^\top \1_n} 
= \frac{1}{1+\lambda}(\mub + \lambda \e_j)$  
for any $j$ and $\lambda > 0$. 
When $\lambda \rightarrow \infty$, the state $j$ tends to become absorbing, that is, $\Delta_{j,:} \rightarrow \e_j^\top\revise{ - G_{j,:}}$, since $\hat \mub_j \rightarrow 1$. Intuitively, we increase the probability of the state to return to itself, making it absorbing at the limit. 
In other words, by modifying \revise{the entries} in one row of $G$, increasing the weight to oneself while reducing proportionally the weights to the others, we increase $\hat \mub_j$, while decreasing $\hat \mub_i$ for $i \neq j$ with the same ratio, $\frac{1}{1+\lambda} = \frac{1}{r^*}$.

\begin{example} \label{ex:example27}
Consider Example~\ref{ex:example2}: 
$$
G=\left[\begin{array}{cccc}  1/2 & 1/4 & 0 & 1/4 \\ 1/4 & 1/2 & 1/4  & 0 \\
 0 & 1/4 & 1/2 & 1/4 \\ 1/4 & 0 & 1/4  & 1/2 \end{array}\right], 
\quad \mub^\top=\left[1/4,1/4,1/4,1/4\right]. 
$$
Choosing  $\hat\mub^\top=\left[\frac25,\frac15,\frac15,\frac15\right]$. We have $\mub^\top./\hat \mub^\top = \left[\frac58,\frac54,\frac54,\frac54\right]$, with optimal parameter $c^*=\frac45$, 
so that $\al^* = \al(c^*)=\left[\frac12, 0, 0, 0\right]^\top$, leading to 
$$ \Delta(\al^*) = 
\left[\begin{array}{cccc} 
1/4 & -1/8 & 0 & -1/8 \\ 
0 & 0 & 0 & 0 \\ 
0 &  0& 0 & 0 \\
0 & 0 & 0 & 0 \end{array}\right]  \; 
 \mathrm{and} \; \ 
 G(\al^*)= \left[\begin{array}{cccc} 
3/4 & 1/8 & 0 & 1/8 \\  
1/4 & 1/2 & 1/4  & 0 \\
 0 & 1/4 & 1/2 & 1/4 \\ 
 1/4 & 0 & 1/4  & 1/2 
\end{array}\right].
$$
This turns out to be a globally optimal solution of the TSDP that minimizes the component-wise $\ell_1$, as proved in the next paragraph. 
\end{example}

\revise{
\begin{remark}
    The subdominant (that is, second largest in modulus) eigenvalue of $G$ is $0.5$, while that of $G(\al^*)$ is $0.65$, indicating an increased mixing time of the associated Markov chain. 
Intuitively, any perturbation causing a substantial increase in the
diagonal entries will slow down time to convergence since a random walk will spend on average more time in each node, hence taking more time to visit all nodes.  
An interesting direction of further research is to look for perturbation that achieve a target distribution while controlling the subdominant eigenvalue. For example, instead of minimizing the norm of the 
 perturbation, one could try to minimize the modulus of the subdominant eigenvalue, or combine these two objectives. 
\end{remark}
}

\paragraph{Optimality of the rank-1 solution $\Delta(\al^*)$} 

As observed in Example~\ref{ex:example27}, it turns out the solution $\Delta(\al^*)$ is optimal for the TSDP in the rank-1 case, under the condition that $j \in \argmax_k \hat \mu_k$, as stated in the following theorem. 

\begin{theorem} \label{lem:optimalrankone}
\revise{Let $G$ be an irreducible stochastic matrix with stationary distribution $\mub>0$}, and let us use the same notation as in Theorem~\ref{th:IminusGsolution}. 
Assume $\hat \mub = \frac{1}{1+\lambda} (\mub + \lambda \e_j)$ for some $\lambda > 0$ and some $j$. 
Assume also that $\lambda \geq \mu_i - \mu_j$ for all $i \neq j$, that is, $j \in \argmax_k \hat \mu_k$. 
Then $\Delta(\al^*)$ where $\al^* = \al(c^*)$ is an optimal solution of the TSDP {for the component-wise and matrix $\ell_1$ norms}. 
\end{theorem}
\begin{proof} 
By~\cite[Theorem~5.1]{BerkhoutHV23}, any feasible rank-one solution of the TSDP without the nonnegativity constraints  has the form 
\[
\Delta(\mathbf x) = \frac{\mathbf x \hat \mub^\top}{\hat \mub^\top \mathbf x} (I-G) \text{ for } \mathbf x \in \mathbb{R}^n, \;
{\mathbf x}^\top \hat \mub \neq 0. 
\] 
In fact, if $\Delta = \mathbf x \mathbf y^\top$ for some vectors $\mathbf x$ and $\mathbf y$,  the constraints $\hat \mub^\top \Delta = 
\hat \mub^\top (I-G)$ leads to 
\[
\hat \mub^\top \mathbf x \mathbf y^\top = 
\hat \mub^\top (I-G) \neq 0 \quad \Rightarrow \quad 
\mathbf y^\top = \frac{\hat \mub^\top}{\hat \mub^\top \mathbf x} (I-G), \text{ where } \hat \mub^\top \mathbf x \neq 0. 
\]
Moreover, \cite[Theorem~5.1]{BerkhoutHV23} proved that there always exists an optimal rank-one  solution to the TSDP without the nonnegativity constraints \revise{for any induced norm} \revise{; this follows from the fact that there exists a rank-one solution achieving the lower bound given in Lemma~\ref{lem:bound} for any $\Delta$ satisfying the constraint $\hat \mub^\top (G + \Delta) = \hat \mub^\top$ (see below for the formula of the rank-one solution for the case of the induced $\ell_1$ norm).} \revise{For the component-wise $\ell_1$ norm, the proof can be found in Appendix~\ref{app:proofL1}.} 

Hence solving $\min_{\mathbf x} \|\Delta(\mathbf x)\|$ solves the TSDP under the rank-one constraint (that is, $\text{rank}(\Delta) = 1$), with the constraint $\Delta \1_n = \1_n$ because $\Delta(\mathbf x)$ satisfies it automatically: $\Delta(\mathbf x) \1_n = \mathbf x \mathbf y^\top \1_n = 0$, but possibly without the nonnegativity constraints. 

Let us develop the above formula when $\hat \mub = \frac{1}{1+\lambda} ( \mub + \lambda e_j)$:  
\[ 
\Delta(\mathbf x) = \frac{\mathbf x (\mub + \lambda e_j)^\top}{(\mub + \lambda e_j)^\top \mathbf x} (I-G)
= 
\lambda 
\frac{\mathbf x   \e_j^\top}{ \revise{\mub}^\top \mathbf x + \lambda x_j} (I-G)
= 
\frac{\lambda}{\mub^\top \x + \lambda x_j}
 \mathbf x  (I-G)_{j,:}, 
\] 
which leads to 
\[
\lambda
\left\|
\frac{\mathbf x   \e_j^\top}{ \mu^\top \mathbf x + \lambda x_j} (I-G)
\right\| 
= 
\frac{\lambda}{|\mub^\top \mathbf x + \lambda x_j|}
\left\| \mathbf x  (I-G)_{j,:}
\right\|. 
\]
Let us start with the matrix $\ell_1$ norm, that is, $\|A\|_{\ell_1} = \max_{\mathbf z} \frac{\|A \mathbf z\|_1}{\|\mathbf z\|_1}$. We have 
\[
\left\| \mathbf x  (I-G)_{j,:}
\right\|_{\ell_1} 
= 
\max_{\mathbf z} \frac{\|\mathbf x  (I-G)_{j,:} \mathbf z\|_1}{\|\mathbf z\|_1}
= 
\|\mathbf x\|_1 \max_{\mathbf z} \frac{\|\mathbf  (I-G)_{j,:} \mathbf z\|_1}{\|\mathbf z\|_1}
= \|\mathbf x\|_1 \|(I-G)_{j,:} \|_\infty. 
\]
For the component-wise $\ell_1$ norm, we have 
\[
\left\| \mathbf x  (I-G)_{j,:}
\right\|_1 
= \|\mathbf x\|_1 \|(I-G)_{j,:} \|_1. 
\] 
Since $\Delta(\mathbf x)$ is independent of the scaling of $\mathbf x$, we can assume w.l.o.g.\ that $\|\mathbf x\|_1 = 1$. Hence  $\min_{\mathbf x} \|\Delta(\mathbf x)\|_*$ where $* \in \{1, \ell_1\}$ is equivalent to solving 
\[
\min_{\mathbf x, \|\mathbf x\|_1 = 1} 
\frac{1}{|\mub^\top \mathbf x+\lambda x_j|} = 
\max_{\mathbf x, \|\mathbf x\|_1 = 1} |\mub^\top \mathbf x+\lambda x_j|, 
\]
which is minimized at $\x^* = \e_i$ for 
$i \in  
\argmax_k \hat \mu_k$. 
(Note that this was proved in \cite[Corollary~5.3]{BerkhoutHV23} for the matrix $\ell_1$ norm.)   
This gives 
\[
\Delta(\x^*) 
= \Delta(\e_i) 
= \frac{\lambda}{\mu_i + \lambda (\e_i)_j}
 \e_i  (I-G)_{j,:} , 
\] 
meaning that the $i$th row of $G$ corresponding to the largest entry in $\hat \mub$ is perturbed in the direction $(I-G)_{j,:}$. Intuitively, the most influential node in $G$ will increase its weight towards the $j$th node while reducing weights to others, proportionally to the $j$th node weights. This makes sense since we are trying to increase the value of $\mu_j$ to $\hat \mu_j$: we use the most influential node to do that.  

Finally, this solution is nonnegative if $\supp (I-G)_{j,:} \subset \supp (I-G)_{i,:}$ which will typically not be the case when $G$ is sparse and $i \neq j$. 
However, if $\lambda \geq \mu_i - \mu_j$ then $\hat \mu_j$ is the largest entry in $\hat \mub$ so that $\x^* = \e_j$ and hence 
\[
\Delta(\mathbf x^*) 
= \frac{1}{\mu_j + \lambda} \e_j (\mub + \lambda \e_j)^\top (I-G)
= \frac{\lambda}{\mu_j + \lambda} \e_j \e_j^\top (I-G). 
\] 
This implies that $G + \Delta (\mathbf x^*) \geq 0$ which is therefore an optimal solution of the TSDP for the component-wise and matrix $\ell_1$ norms\footnote{This case is analyzed in~\cite[Section~7]{BerkhoutHV23}. However, authors do no mention the condition $\lambda \geq \max_i \mu_i - \mu_j$ to obtain~\cite[Equation~(7.2)]{BerkhoutHV23}.}.  
Recall that $\Delta(\al^*) = \alpha_j \e_j \e_j^\top (I-G)$, see~\eqref{eq:formrank1}, and $\alpha_j$ was minimized to obtain the best possible solution of the TSDP of this form. Since the above rank-one solution is optimal, this means that $\Delta(\mathbf x^*)  = \Delta(\al^*)$. In fact, noting that  
$r^* = \frac{1}{1+\lambda}$ and 
$r_* = \frac{\mu_j}{\hat \mu_j} 
= \frac{\mu_j}{\hat \mu_j} 
= \frac{(1+\lambda)\mu_j}{\mu_j + \lambda}$, 
we have 
\[
\alpha_j 
= 
1 - \frac{r_*}{r^*} 
= 
1 - \frac{\frac{(1+\lambda) \mu_j}{\mu_j + \lambda}}{1+\lambda} 
=  
\frac{\lambda}{\mu_j + \lambda}, 
\] 
which concludes the proof. 
\end{proof} 

\begin{example}
Let us take the final matrix from Example~\ref{ex:example27}: 
\[ 
G = 
\left[\begin{array}{cccc} 
3/4 & 1/8 & 0 & 1/8 \\  
1/4 & 1/2 & 1/4  & 0 \\
 0 & 1/4 & 1/2 & 1/4 \\ 
 1/4 & 0 & 1/4  & 1/2 
\end{array}\right] 
\quad \text{ with } \quad
\mub = \left[\frac25,\frac15,\frac15,\frac15\right].
\] 
Let $\hat \mub = \frac{1}{1+\lambda} (\mub + \lambda \e_2) = \left[\frac{4}{11},\frac{3}{11},\frac{2}{11},\frac{2}{11}\right]$ for $\lambda = \frac{1}{10}$, so that $\hat \mu_1 >   \hat \mu_2$. We have 
\[
\Delta(\mathbf x^*) = 
\left[\begin{array}{cccc} 
 -11/160   &      11/80      &   -11/160 & 0 \\  
0 & 0 & 0  & 0 \\
 0 & 0 & 0 & 0 \\ 
 0 & 0 & 0  & 0 
\end{array}\right]  
\neq 
\Delta(\al^*) = 
\left[\begin{array}{cccc} 
  0  &      0      &   0 & 0 \\  
 -1/12       &    1/6     &      -1/12 & 0 \\
 0 & 0 & 0 & 0 \\ 
 0 & 0 & 0  & 0 
\end{array}\right], 
\]
while 
\[ 
\Delta^* = \argmin_{\Delta \in \mathcal{D}} \|\Delta\|_1 = 
\left[\begin{array}{cccc} 
 -1/16    &       1/16        &   0 & 0 \\  
 0       &    1/12     &     -1/12 & 0 \\
 0 & 0 & 0 & 0 \\ 
 0 & 0 & 0  & 0 
\end{array}\right], 
\]
  where  $\| \Delta(\mathbf x^*) \|_1 = \frac{11}{40} < \| \Delta^* \|_1 =  \frac{7}{24} < \| \Delta(\mathbf x^*) \|_1 = \frac{1}{3}$, but \revise{$\Delta(\mathbf x^*)$} is not admissible as $G+\Delta(\mathbf x^*) \ngeq 0$. Note that $\Delta^*$ was computed using linear optimization; see section~\ref{sec:linopt}. 
\end{example}


\subsection{Ordering of $\mub$}  \label{sec:reorder}

\revise{
In this section we show how to reduce the difference  $r^*-r_*$  by introducing a permutation on the vector
$\hat \mub$, and hence the norm of $\Delta(\al^*)$ is expected to be smaller; see Equation~\eqref{eq:boundDal}.  
In other words, since $\hat \mub$ is a parameter to be chosen by the user, it is interesting to know in which situation $\Delta(\al^*)$ is expected to have a smaller norm to achieve the target $\hat \mub$. 
To do so, we derive a result on the orderings of the stationary distributions $\mub$ and $\hat\mub$ that  minimizes $r^*-r_*$. } 
\begin{theorem} \label{th:ordering}
Let $P$ and $\hat P$ be $n\times n$ permutation matrices that {\em reorder} the elements of 
$\mub$ and $\hat \mub$, respectively, meaning that the elements of the permuted vectors $P\mub$ and $\hat P\hat\mub$, each form a non-decreasing series. Then
$$  [\min_i (P\mu)_i/ (\hat P\hat\mu)_i , \max_i (P\mu)_i/(\hat P\hat\mu)_i ] \subset [ \min_i \mu_i/\hat\mu_i , \max_i \mu_i/\hat\mu_i ]. 
$$ 
\end{theorem}
\begin{proof} See Appendix~\ref{sec:appendixA}. 
\end{proof}
This theorem shows that when the vectors $P\mub$ and $\hat P\hat\mub$ have ``coherent" orderings, then the corresponding 
interval $[r_*(P\mub,\hat P\hat\mub),r^*(P\mub,\hat P\hat\mub)]$ is minimal, and hence the norm $\|\al(c^*)\|$ \revise{is expected to be smaller}. 
Note that the permutation of $\mub$ can be achieved by a permutation on the rows of the matrix $G$, but this is {\em not} a similarity transformation. 


\revise{In summary}, our result indicates that
the required perturbation $\Delta$ \revise{is expected to be smaller} when $\hat\mub$ is coherently ordered with $\mub$. 




\section{Formulation with linear optimization}  \label{sec:linopt} 

In the previous section, we focused on solutions of the TSDP of the form $\Delta(\al) = \diag(\al)(I_n - G)$ with support in $\supp(G+I_n)$. 
In this section, we propose an efficient linear optimization formulation of the TSDP with support constraint and for the component-wise $\ell_1$ norm. 
This allows us to solve large problems efficiently, especially when $\Omega$ or $G$ are sparse. 

\subsection{Linear optimization formulation} \label{sec:linoptformu}

The TSDP with support constraint $\supp(\Delta)\subset \Omega$ 
and for the component-wise $\ell_1$ norm can be formulated as follows: 
\begin{align} 
\min_{\Delta \in \mathbb{R}^{n \times n}}  \| \Delta \|_1  
\quad \text{ such that } \quad & \Delta \1_n = \0_n,  \nonumber \\ 
& \hat \mub^\top \Delta = \hat \mub^\top (I-G), \label{eq:sparseTSDP} \\ 
& \Delta + G \geq 0, \nonumber \\ 
& \Delta_{i,j} = 0 \text{ for } (i,j) \notin \Omega.  \nonumber 
\end{align} 
We will refer to this problem as $\ell_1$-TSDP with support constraint. 
Since the entries of $\Delta$ are non-zero only in the  set $\Omega$, \eqref{eq:sparseTSDP} has only $|\Omega|$ variables (the zero entries of $\Delta$ are not variables).  
In the dense case, that is, when $\Omega = \{(i,j) \ | \ 1 \leq i, j \leq n\}$, a naive way to formulate~\eqref{eq:sparseTSDP} is to introduce $n^2$ variables, say $Z \in \mathbb{R}^{n \times n}_+$, one for each entry of $\Delta$, impose $Z \geq \Delta$ and $Z \geq -\Delta$ and minimize $\sum_{i,j} Z_{i,j}$; this is the standard way to linearize the $\ell_1$ norm.  However, on top of introducing $n^2$ variables, it introduces $2n^2$ inequalities which is highly ineffective, since solving linear optimization problems requires, roughly speaking, a cubic number of operations in the number of variables or in the number of constraints\footnote{More precisely, it requires $\tilde{O}(n^\omega \log(n/\delta))$ time where $O(n^\omega)$ is the time to multiply two $n$-by-$n$ matrices, 
$\delta$ is the relative accuracy, and $\tilde O$ ignores polylog$(n)$ factors.} (solving the dual); see~\cite{van2020deterministic} and the references therein.  
For example, the problem was formulated in this way in~\cite[Equation (8.1)]{BerkhoutHV23}, and the formulation of the authors using a commercial solver, Gurobi, could not solve problems larger than $n=500$ within 10 minutes. 
The reason is that \cite{BerkhoutHV23} did not consider the component-wise $\ell_1$ norm, but matrix norms for which it is not possible to avoid the introduction of $O(|\Omega|)$ inequalities to model the problem via linear optimization. 


 Let us reformulate the problem in a more cost-effective way. 
First, it is interesting to note that whenever $G_{i,j} = 0$, we have $\Delta_{i,j} \geq 0$ because of the constraints $G + \Delta \geq 0$, 
and hence there is no need to introduce an additional variable for $\Delta_{i,j}$ when $G_{i,j} = 0$ to minimize $|\Delta_{i,j}|$. 
Let us denote the support of $G$ as $\supp(G)$, its complement as $\xoverline[0.9]{\supp(G)}$  (that is, the set of \revise{index pairs} corresponding to zero entries of $G$), and let us decompose $\Delta$ using three nonnegative terms: 
$\Delta = \Delta^0 + \big( \Delta^+ - \Delta^- \big)$, where 
$$
\supp(\Delta^0)  
\;  \subset \;  
\Omega^0 := \Omega \cap \xoverline[0.9]{\supp(G)}, 
$$ 
and 
$$
\supp(\Delta^+), \supp(\Delta^-)  
 \;  \subset \; 
 \Omega^{\pm} := \Omega \cap \supp(G). 
$$  
Note that $\Omega^{0}$ and $\Omega^{\pm}$ is a partition of $\Omega$:  $\Omega^{\pm} \cup \Omega^{0} = \Omega$ and $\Omega^{\pm} \cap \Omega^{0}  = \emptyset$. 
Hence 
$\Delta^0(i,j) \geq 0$ can be non-zero  when $(i,j) \in \Omega$ and $G_{i,j} = 0$, 
while $\Delta^+(i,j), \Delta^-(i,j) \geq 0$ can be non-zero when $(i,j) \in \Omega$ and $G_{i,j} > 0$. 
We will denote the number of non-zero entries of $\Delta^0$ by $p^0 = \Big|\Omega \cap \xoverline[0.9]{\supp(G)} \Big| \leq |\Omega|$, and the number of non-zero entries in $\Delta^+$ and $\Delta^-$ by $p^{\pm} = |\Omega \cap \supp(G)|$. 
With these new variables, the TSDP~\eqref{eq:sparseTSDP} can be reformulated as follows 
\begin{align}
\min_{\Delta^0, \Delta^+, \Delta^- \in \mathbb{R}^{n \times n}_+} 
& \sum_{i,j} \Delta_{i,j}^0 + \Delta^+_{i,j} + \Delta^-_{i,j}  \nonumber \\ 
\text{ such that } & \big( \Delta^0 + \Delta^+ - \Delta^- \big) \1_n = \0_n,    \label{eq:sparseTSDPv2} \\ 
& \hat \mub^\top \big( \Delta^0 + \Delta^+ - \Delta^- \big) = \mathbf z := \hat \mub^\top (I-G), \nonumber \\ 
& \Delta^0 \leq \1_{n,n}, \;
\Delta^- \leq G, \; \Delta^+ \leq \1_{n,n}-G, 
\nonumber \\ 
& \supp(\Delta^0) \subset \Omega^0, \;
\supp(\Delta^+), 
\supp(\Delta^-)  
\subset \Omega^{\pm}.  \nonumber 
\end{align}
The constraint  $\Delta^- \leq G$ ensures that $G + \Delta$ is nonnegative. 
Since $G + \Delta$ is also stochastic, as $\Delta \1_n = 0$, the constraints $\Delta^0 \leq \1_{n,n}$  and $\Delta^+ \leq \1_{n,n}-G$ are redundant. We have not observed much impact on the solver by removing or using these upper bound constraints. 
 Note that we are minimizing the sum of the entries of $\Delta^{+}$ and $\Delta^-$, and hence  we will always have at optimality that the supports of $\Delta^{+}$ and $\Delta^-$ are disjoint.

In any case, \eqref{eq:sparseTSDPv2} has only $2n$ equalities and less than $2 |\Omega|$ bounded variables. 
The fact that we have only $2n$ equalities allows us to solve this problem in roughly $O(n^3)$ operations via the dual. This is a significant improvement. Moreover, each inequality has at most $n$ variables, and hence this is a very sparse linear optimization problem, even sparser when the set $\Omega$ is sparse.  This will allow commercial solvers, such as Gurobi, to solve such problems much faster than the worst case $O(n^3)$ operations; for example, we will be able to solve problems with $n = 10^5$ with $\Omega$ having a few non-zero entries per row in a few minutes; see section~\ref{sec:Num} for experiments.

\subsection{Sparsity of $\ell_1$ norm minimization}

The optimal solution of~\eqref{eq:sparseTSDPv2} will have some degree of sparsity. In this section, we make this statement precise by providing an explicit upper bound on the number of non-zero entries in an optimal solution; see Theorem~\ref{lem:sparsitysol} below. 

Let us rewrite~\eqref{eq:sparseTSDPv2} in a more standard form, by vectorizing the non-zero entries of $\Delta^0$ in $\mathbf x^0 \in \mathbb{R}^{p^0}$,  of $\Delta^+$ in $\mathbf x^+ \in \mathbb{R}^{p^\pm}$, and of $\Delta^-$ in $x^- \in \mathbb{R}^{p^\pm}$, and let $\mathbf x = (\mathbf x^0,\mathbf x^+,\mathbf x^-) \in \mathbb{R}^p$:  
\begin{align}
\min_{\mathbf x = (\mathbf x^0, \mathbf x^+, \mathbf x^-)  \in \mathbb{R}^p}   \1_n^\top \mathbf x   
\quad \text{ such that } \quad  & A \mathbf x = \mathbf b, \; \mathbf x \geq 0, \; \mathbf  x^- \leq \mathbf g,    \label{eq:sparseTSDPv3}  
\end{align}
where $p = p^0 + 2p^\pm$, $A \in \mathbb{R}^{2n \times p}$  contains the coefficient of the $2n$ linear equalities and has two non-zero entries per column (since each variable appears in exactly two  constraints: $\Delta_{i,j}$ appears in $\Delta(i,:) \1_n = 0$ and $\mub^\top \Delta(:,j) = \mathbf z$), 
$\mathbf b = [\0_n; \mathbf z]$ is the right-hand side, and 
$\mathbf g \in \mathbb{R}^{p^\pm}$ contains the non-zero entries of $G$ in the \revise{set} $\Omega$. 

\begin{theorem} \label{lem:sparsitysol} 
If~\eqref{eq:sparseTSDPv3} is feasible, there exists an optimal solution, $\mathbf x^*$, of~\eqref{eq:sparseTSDPv3} such that 
\[
|\supp(\mathbf x^*)| \; \leq \; \min\big(|\Omega|,|\supp(G\cap \Omega)| + 2n\big). 
\] 
Equivalently, if the $\ell_1$-TSDP with support constraint~\eqref{eq:sparseTSDP} is feasible, there exists an optimal solution of~\eqref{eq:sparseTSDP}, $\Delta^*$, such that 
$$
|\supp(\Delta^*)| \leq \min\big(|\Omega|,|\supp(G \cap \Omega)| + 2n\big).
$$
\end{theorem}
\begin{proof} 
The feasible set of~\eqref{eq:sparseTSDPv3}, $\mathcal{P} = \{  \mathbf x = (\mathbf x^0, \mathbf x^+, \mathbf x^-)  
  \ | \ 
A\mathbf x = \mathbf b, x \geq 0, \mathbf x^- \leq \mathbf g \}$,   
is bounded, and hence is a polytope. Therefore, by the fundamental theorem of linear optimization, if $\mathcal{P}$ is non-empty, that is,~\eqref{eq:sparseTSDPv3} is feasible, 
there exists an optimal solution which is a vertex of  $\mathcal{P}$, say $\mathbf x^*$. 

Let us first show $|\supp(\mathbf x^*)| \leq |\Omega|$. 
For any optimal solution, $\mathbf x = (\mathbf x^0,\mathbf x^+,\mathbf x^-)$, we must have $\mathbf x^+_i = 0$ or $\mathbf x^-_i = 0$ for all $i$, 
otherwise we can strictly decrease the objective. This is because the entries of $\mathbf x^+ - \mathbf x^-$ are equal to the entries of $\Delta$ in $\supp(G) \cap \Omega$. This means that $|\supp([\mathbf x^+, \mathbf x^-])| \leq p^\pm = |\supp(G) \cap \Omega|$, and hence $|\supp(\mathbf x)| \leq |\Omega|$.  

It remains to prove that $|\supp(\mathbf x^*)| \leq |\supp(G)| + 2n$. 
Recall that any vertex of a $p$-dimensional polytope must have $p$ linearly independent active constraints. 
Since there are (only) $2n$ equalities in~\eqref{eq:sparseTSDPv3}, there must be $p-2n$ of the bound constraints active at any vertex of $\mathcal{P}$. Let $\mathbf x = (\mathbf x^0,\mathbf x^+,\mathbf x^-)$ be a vertex of $\mathcal{P}$. 
Each entry in $\mathbf x^+ \geq 0$ and $0 \leq \mathbf x^- \leq \mathbf g$ can touch at most one bound constraint (the entries in $\mathbf x^-$ cannot touch two of them, since by definition $\mathbf g > 0$). This is possible if $x^+_i = 0 = x^-_i$, or if $x^+_i = 0$ and $x^-_i = g_i > 0$. 
This means that $x^0$ must touch at least $\revise{p} - 2n - 2p^\pm = p^0 - 2n$ of the lower bound constraints, $x^0 \geq 0$, that is, $x^0$ must have at least $p^0 - 2n$ entries equal to zero. 
This implies that $|\supp(\mathbf x^0)| \leq 2n$, and gives the result, since $|\supp([\mathbf x^+, \mathbf x^-])| \leq |\supp(G \cap \Omega)|$ at optimality; see above. 
\end{proof}

Theorem~\ref{lem:sparsitysol} implies that if $G$ or $\Omega$ is sparse, then there is an optimal solution of~\eqref{eq:sparseTSDP} which is sparse.  
To attain the upper bound $|\supp(G \cap \Omega)|+2n$, $\Delta$ must be negative in all entries of $\supp(G)$ as it requires $\Delta^- = G$. 
This is unlikely to happen at optimality since we are minimizing the sum of these  variables. 
In fact, in all cases we have tested, optimal solutions are significantly sparser; see the numerical experiments in section~\ref{sec:Num}.  

\subsection{Efficient algorithms via a column-generation approach} 

The number of columns of $A$ in~\eqref{eq:sparseTSDPv3} is $O(|\Omega|)$. If 
$\Omega$ is sparse, we can formulate the problem directly and solve it very efficiently via commercial solvers; see section~\ref{sec:Num} for numerical experiments. 

However, if $\Omega$ is dense, for example $\Omega = \{(i,j) \ | \ 1 \leq i,j \leq n\}$, constructing the $O(n^2)$ columns of the constraint matrix $A$ is time consuming, even if $A$ is sparse (2 non-zeros per column). 
This is particularly wasteful if $G$ is sparse since we know only a few entries in the optimal solution $\Delta$ will be non-zero, namely at most $|\supp(G)| + 2n$ (Theorem~\ref{lem:sparsitysol}). 

This calls for a column-generation approach: 
First solve the problem for a well-chosen sparse $\Omega$. 
Then progressively add \revise{index pairs} in $\Omega$ that will reduce the objective function. 

\paragraph{Choice of initial $\Omega$} 
\revise{Given a set $\Omega$, deciding whether one can change the stationary distribution from $\mub$ to $\hat \mub$ has been studied in~\cite{brualdi1968convex, breen2015stationary}. However, finding a small initial support such that the problem is feasible, while avoiding a horrendous enumeration, appears to be a non-trivial question, to the best of our knowledge.}  
Luckily, we know that $\Omega = \supp(G+I)$ is always feasible (Theorem~\ref{th:IminusGsolution}), and hence we will initialize $\Omega$ with the support of $G+I$. 
For $G$ dense, it is unclear how to initialize $\Omega$ with a sparse matrix to make the problem feasible, and this is probably not as useful as the optimal solution is not necessarily very sparse.


\paragraph{Choice of entries to add to $\Omega$} 
To add \revise{index pairs} in $\Omega$ that will reduce the objective, we can resort to the well-known column-generation approach in linear optimization, see~\cite{nemhauser2012column} and the references therein.   
Let us recall the basic idea of column generation. We are trying to solve a linear optimization problem of the form 
\[
\min_{\mathbf x = (\mathbf x_1,\mathbf x_2) \in \mathbb{R}^{p_1 \times p_2}} 
\mathbf c_1^\top \mathbf x_1 + \mathbf c_2^\top \mathbf x_2 
\quad \text{ such that } \quad 
[A_1, A_2] \left[ \begin{array}{c}
     \mathbf x_1  \\
     \mathbf x_2 
\end{array} \right] 
 = \mathbf b, \;
\mathbf x \geq 0, 
\] 
where $p = p_1+p_2$, $A_i \in \mathbb{R}^{m \times p_i}$, 
$\mathbf b \in \mathbb{R}^m$, 
$\mathbf c_i \in \mathbb{R}^{p_i}$ for $i =1,2$. The dual is given by 
\[
\max_{\mathbf y \in \mathbb{R}^m} 
\mathbf b^\top \mathbf y \quad \text{ such that } 
\quad A_1^\top \mathbf y \leq \mathbf c_1 \text{ and }   A_2^\top \mathbf y \leq \mathbf c_2. 
\] 
Assume we solve the primal without the variables $\mathbf x_2$ so that the constraint $A_2^\top \mathbf y \leq \mathbf c_2$  disappears from the dual, 
and obtain $(\mathbf x_1^*,\mathbf y^*)$ as an optimal primal-dual solution. 
If $A_2^\top \mathbf y^* \leq \mathbf c_2$, that is, the dual solution, $\mathbf y^*$, remains feasible with the additional variables $\mathbf x_2$ being taken into account, then $(x_1^*,0)$ is an optimal solution to the original problem, by strong duality. If this is not the case, that is, $A_2(:,i)^\top \mathbf y^* > c_2(i)$ for some $i$, then allowing the variable $x_2(i)$ to enter the solution will allow us to reduce the objective (this is the rationale of a pivoting step in the simplex algorithm). The quantity $\mathbf c_2 - A_2^\top \mathbf y^*$ are  the so-called reduced costs of the variables $\mathbf x_2$ which needs to be nonnegative at optimality.   


Let us apply this idea to our problem. 
Let $\mathbf y^0 \in \mathbb{R}^{n}$ be the dual variables associated to the constraints $\Delta \1_n = 0$, and 
$\mathbf y^\mu \in \mathbb{R}^{n}$ to the constraints $\Delta^\top \hat \mub  = \mathbf z$. 
Since, at the \revise{first} step, $\Omega = \supp(G+I)$, 
we will add entries with index $(i,j) \notin \Omega$ such that $G_{i,j} = 0$. 
Adding $(i,j)$ in $\Omega$, that is, adding a variable in the primal corresponding to $\Delta^0_{i,j}$, adds a constraint in the dual, namely 
$y^0_i + \hat \mu_j y^\mu_j  \leq 1$. In fact, in the primal, the coefficient of the variable $\Delta^0_{i,j}$ is only non-zero for two constraints: $\Delta_{i,:} \1_n = 0$ and $\Delta_{:,j}^\top \hat \mub = z_j$.  
If this dual constraint is violated, that is,  $y^0_i + \hat \mu_i y^\mu_j  > 1$, this means that the current primal solution is not optimal for $\Delta^0_{i,j}=0$ (the reduced cost are not nonnegative for the current basis). 
Hence adding variables $(i,j)$ in $\Omega$ for which $y^0_i + \hat \mu_i y^\mu_j > 1$ will reduce the objective function value.  However, it is not possible for a generic solver to perform such a task without generating explicitly all constraints in the dual (and this takes time when $\Omega$ is dense); this is because the solver cannot guess the structure of our problem. 

To take advantage of these observations, 
we adopt the following strategy: 
We initialize $\Omega^0 = \supp(G+I)$. 
We solve~\eqref{eq:sparseTSDP}, 
and then we add to $\Omega$ the $|\Omega^0|$ largest entries from the rank-two matrix $\mathbf y^0 \1_n^\top + \hat \mub (\mathbf y^\mu)^\top$ that are larger than 1. We provide the commercial solver, Gurobi, with the previous basis, for a clever initialization. 
We stop adding new variables when the decrease\footnote{\revise{The error can only decrease between two iterations since the support of the optimal solution of~\eqref{eq:sparseTSDP} at iteration $i+1$, $\Omega^{i+1}$, contains the support from the previous iteration, $\Omega^{i}$; see step~7 in Algorithm~\ref{alg:CGtsdp}.}} in the relative error between two iterations is smaller than some prescribed accuracy $\delta$, that is, 
$\frac{\|\Delta^{\text{old}}\|_1 - \|\Delta^{\text{new}}\|_1}{\|G\|_1} < \delta$, or if the CG has converged, that is, $\mathbf y^0 \1_n^\top + \hat \mub (\mathbf y^\mu)^\top \leq \1_{n,n}$. 
We will use $\delta = 1\%$ and $\delta = 0.01\%$ in the numerical experiments.  

Algorithm~\ref{alg:CGtsdp} summarizes our strategy, which is a batch column-generation approach. It is more effective in MATLAB because adding only a single variable at each iteration would lead to many calls of the Gurobi solver and would be ineffective. Ideally, one should implement the strategy within Gurobi, in C, but this is out of the scope of this paper.

\algsetup{indent=2em}
\begin{algorithm}[ht!]
\caption{Column-Generation based algorithm for the TSDP~\eqref{eq:sparseTSDP} \label{alg:CGtsdp}} 
\begin{algorithmic}[1] 
\REQUIRE 
A stochastic irreducible matrix $G \in \mathbb{R}^{n \times n}$, 
a \revise{positive} target distribution $\hat \mub \in \revise{\mathbb{R}^{n}_{+}}$, 
a set $\Omega$ of interest (default: $\{(i,j) \ | \ 1 \leq i,j \leq n\}$), 
a target relative accuracy $0 \leq \delta < 1$ (default: $10^{-4}$). 
    \medskip  

\STATE $\text{obj}_{-1} = +\infty$, 

\STATE $\Delta^0 = \Delta(\al^*)$, 
$\text{obj}_{0} = ||\Delta^0||_1$

\STATE Initialize $\Omega^1 = \supp(G+I) \cap \Omega$, $n^+ = |\Omega^1|$, $\Omega^+ = \Omega^1$, $i = 1$. 

\textbf{Note}: $\Omega^1$ might not be feasible if $\Omega$ does not contain $\supp(G+I)$. 

\WHILE{ $\text{obj}_{i-2} - \text{obj}_{i-1} > \delta \|G\|_1$ and $\Omega^+ \neq \emptyset$ }

	\STATE Solve~\eqref{eq:sparseTSDP} with $\Omega = \Omega^i$ starting from $\Delta^{i-1}$, 
 and obtain the optimal solution $\Delta^i$ with   optimal dual variables $\mathbf y^0$ and $\mathbf y^\mu$, 
 let    $\text{obj}_{i} = ||\Delta^i||_1$. 

\STATE \label{alg:step6} Construct the matrix $R = \mathbf y^0 \1_n^\top + \hat \mub (\mathbf y^\mu)^\top - \1_{n,n}$, and let $\Omega^+$ be the set of \revise{index pairs} in $\Omega$ corresponding to the $n^+$ largest positive entries in $R$. 

\textbf{Note}: $R$ could have less than $n^+$ positive entries, in which case we only add the positive ones. In fact, if $\Delta^i$ is globally optimal, $\Omega^+ = \emptyset$. 

 \STATE Set $\Omega^{i+1} = \Omega^{i} \cup \Omega^+$, 
 $i \leftarrow i+1$. 
 
\ENDWHILE

\end{algorithmic}  
\end{algorithm}


\paragraph{Computational cost and key improvement in the  selection strategy}  

Solving~\eqref{eq:sparseTSDP} requires to solve linear optimization problems with $2n$ constraints. This requires roughly $O(n^3)$ operations in the worst case; see section~\ref{sec:linoptformu}.  
However, we will see that that it is empirically much faster as this linear program is very sparse, especially when $G$ is; see section~\ref{sec:salabsynt} for numerical experiments.  

In terms of memory, the constraint matrix $A \in \mathbb{R}^{2n \times O(|\Omega|)}$ actually only requires $O(|\Omega|)$ memory as there are two non-zeros per column of $A$. 
Empirically, we will see that a few iterations of Algorithm~\ref{alg:CGtsdp} (in most cases, less than 10) are enough to obtain good solutions. 
Hence $|\Omega| = O(\nnz(G))$, and the memory requirement is proportional to that of storing $G$. 

For large $n$, we noticed that constructing explicitly the rank-two matrix $T = \mathbf y^0 \1_n^\top + \hat \mub (\mathbf y^\mu)^\top$ to extract the \revise{index pairs} corresponding to the largest entry was a bottleneck in the algorithm (step~\ref{alg:step6}): it requires $O(n^2)$ memory and $O(n^2 \log(n^2))$ operations (to sort the entries of $T$).   
This rank-two matrix has positive factors in each rank-one term, namely $\1_n$ and $\hat \mub$, and hence the largest entries of $T$ depend directly on the largest entries of $\mathbf y^0$, $\mathbf y^\mu$ and $\hat \mub$. 
We have designed a simple heuristic, when $n$ is large, to tackle this problem: for a parameter $m \leq n/2$, 
extract the $m$ largest indices in $\mathbf y^0$ and the $m$ largest indices in $\hat \mub$ and put them in $\mathcal{I}$ (possibly remove duplicates),  and extract the $2m$ largest indices in $\mathbf y^\mu$ and put them in $\mathcal{J}$. Then perform an exhaustive search on the submatrix $T(\mathcal{I},\mathcal{J})$ which requires $O(m^2 \log(m^2))$ operations. We used \revise{$m = 200$} in our implementation, and use this heuristic as soon as $n > 200$. This means that, when $n$ is large, we add at most $m^2=40,000$ entries in $\Omega$ at each step. 
Note that more sophisticated heuristics exist\footnote{We actually tried the heuristic from~\cite{higham2016estimating}, but it does not scale well to extract a large number of \revise{index pairs} as it uses a dense $n$-by-$p$ matrix, where $p$ is the number of \revise{index pairs} to extract. In our case, $p$ is of the order of $n$ making this heuristic require $O(n^2)$ memory and operations, which is what we need to avoid.}, e.g.,~\cite{higham2016estimating} and \cite{ballard2015diamond}. 
However, our current heuristic performs  extremely well on all the tested cases. For example, we tested against the exhaustive search for the real data set moreno (with $n = 2155$; see section~\ref{sec:realdata}): Algorithm~\ref{alg:CGtsdp} with $\delta = 0$ (that is, an optimal solution is sought) took 41 seconds with an exhaustive search in step~\ref{alg:step6}, and only 5.6 seconds with our heuristic, to obtain the same optimal objective function value.   



\section{Numerical experiments}  \label{sec:Num}

We will compare the following solutions: 
\begin{itemize}

\item \revise{$MH$: the Metropolis-Hastings algorithm; see section~\ref{sec:intro} and in particular equation~\eqref{eq:MHalso}.}   

    \item $D$: the closed-form solution $\Delta(\al^*)$ described in Theorem~\ref{th:IminusGsolution}.   This solution is computed extremely fast, requiring $O(\nnz(G))$ operations.

\item $S$: The optimal solution of~\eqref{eq:sparseTSDP} with $\Omega = \supp(G+I)$. $S$ stands for sparse. 

\item $GS$: The optimal solution of~\eqref{eq:sparseTSDP} with $\Omega = \{(i,j) \ | \ 1 \leq i,j \leq n\}$ computed by solving directly~\eqref{eq:sparseTSDP}. $GS$ stands for global solution.  

\item $CG(\delta)$: The solution of~\eqref{eq:sparseTSDP} with $\Omega = \{(i,j) \ | \ 1 \leq i,j \leq n\}$ computed with Algorithm~\ref{alg:CGtsdp} 
with stopping criterion $\delta$; 
we will use $\delta \in \{ 10^{-2}, 10^{-4}, 0\}$. 
$CG$ stands for column generation. 
The case $\delta = 0$ should generate a solution with the same objective as $GS$, but hopefully significantly faster when $G$ is sparse.  
Note that choosing $\delta$ sufficiently large generates the solution $S$ since Algorithm~\ref{alg:CGtsdp} is initialized with $\Omega = \supp(G+I)$. 
    
\end{itemize} 
By construction, we know in advance that, in terms of objective function values, we have 
\[
     \| D \|_1 
 \geq 
 \|S\|_1 
 \geq
  \| CG(\delta) \|_1 
 \geq 
  \| CG(0)  \|_1  
 =  
  \|GS\|_1. 
\]
In fact, $S$ is the optimal solution for $\Omega = \supp(G+I)$, $D$ is a feasible solution for that support, 
and $CG(\delta)$ is initialized with $\Omega = \supp(G+I)$. 
It will be interesting to compare the gaps between these objectives, and the computational times. 
\revise{Moreover, $\|MH\|_1 \geq \| S \|_1$ since the support of $MH$ is contained in that of $G+I$; see section~\ref{sec:intro}.}

We do not compare with the algorithms in~\cite{BerkhoutHV23} because they are not competitive. For example, for the 3 real data sets they consider, their fastest approach able to compute a feasible solution in all cases (namely, the rank-1 steps heuristic--R1SH) takes 
    228, 50 and 75 seconds, respectively; 
    see~\cite[Table~7]{BerkhoutHV23}. 
    The solution we will generate with $CG(0)$ is computed faster (less than 6 seconds in all cases), is sparse, and is globally optimal; see Tables~\ref{tab:resultsreal001}, \ref{tab:resultsreal01} and~\ref{tab:resultsreal05}.  
Moreover, the solutions in~\cite{BerkhoutHV23} do not minimize the component-wise $\ell_1$ norm, but matrix norms, and do not produce sparse solutions; in fact, they produce dense low-rank solutions. 
To cite their own words, 
\begin{quote} 
``The applicability of our heuristics for sparse $G$ is hindered. Being able to perturb only one row in a stochastic matrix, it is not hard to imagine
 that the number of reachable stochastic matrices is limited. In other words, finding a
 rank-1 perturbation towards a specific stationary distribution (the main focus of this paper) is often infeasible.'' 
\end{quote}

 \paragraph{Quality measures} We will report the following three quantities for each solution $\Delta$:  
\begin{itemize}
    \item \textbf{obj}: the relative objective value $\frac{\|\Delta\|_1}{\|G\|_1}$,

    \item  \textbf{spars}: the relative sparsity $\frac{|\supp(\Delta)|}{|\supp(G+I)|}$, and 

    \item \textbf{time}: the computational time in seconds. 
\end{itemize}

 \paragraph{Softwares and machine}   
All experiments are performed on a laptop with processor 12th Gen Intel(R) Core(TM) i9-12900H  2.50 GHz, 32Go RAM, on MATLAB R2019b.  
The  code is available from \url{https://gitlab.com/ngillis/TSDP}. To solve the linear optimization problems, we use the commercial solver Gurobi\footnote{Although it is a closed-source commercial software, see  \url{https://www.gurobi.com/}, free use is possible for academic users.}, version 10.00.

\subsection{Synthetic data sets} \label{sec:synth}

In this section, we mostly focus on the run time of the proposed algorithms depending on the size and sparsity of the input matrix, and the gap 
 in the objective function between the different approaches. This will allow us to estimate, empirically, their computational cost and distance to optimality.  

Getting inspiration from queuing matrices, where each node is connected only to its two neighbors, one on the left and one on the right, 
 we generate irreducible matrices as follows: Given a parameter $k$, each node is connected to its $k$ neighbors on the right and  to its $k$ neighbors on the left. Each non-zero entry is randomly generated using the uniform distribution in [0,1], and $G$ is then normalized to make it stochastic. 
Hence each row of $G$ has $2k$ non-zero entries (except for the first $k$ and last rows $k$ that do not have enough left and right neighbors, respectively). 
Here is an example of such a matrix generated randomly with parameters $n=5$ and $k=2$ (and two digits of accuracy): 
\[
G = \left[ \begin{array}{ccccc} 
 0 &  0.32 &  0.68 &  0 &  0 \\ 
 0.18 &  0 &  0.39 &  0.43 &  0 \\ 
 0.26 &  0.32 &  0 &  0.30 &  0.12 \\ 
 0 &  0.33 &  0.32 &  0 &  0.35 \\ 
 0 &  0 &  0.28 &  0.72 &  0 \\ 
\end{array} \right]. 
\] 
We set the vector\footnote{We also made experiments with $\hat \mub = \frac{\1_n}{n}$. It produces somewhat similar results, but the differences between the solutions $CG(\delta)$ were slightly less pronounced, and hence we chose to report the results for $\hat \mub =  G^\top \frac{\1_n}{n}$.} 
$\hat \mub =  G^\top \frac{\1_n}{n}$, 
that is, we initialize $\hat \mub$ with equal probability for all nodes, then apply once step of the power iteration. This makes $\hat \mub$ somewhat closer to $\mub$.   
\revise{Note that the sparsity pattern of $G$ is symmetric, and hence $MH$ is guaranteed to generate a feasible solution.} 

\revise{
\begin{remark}[Choice of $\hat \mub$] \label{rem:choicemuh}
There are plenty of other possible choices for $\hat \mub$, and ours here is somewhat arbitrary. 
It balances the importance of the nodes in the graph, as $\hat \mub$ is close to $\frac{\1_n}{n}$. For example, in social networks, this corresponds to make everyone have a similar popularity, and in road networks to make all roads equally congested. 

Another example, that we considered 
in section~\ref{sec:rank1}, is the choice $\hat \mub 
= \frac{1}{1+\lambda}(\mub + \lambda \e_j)$  
for some $j$ and $\lambda > 0$. This is an interesting perturbation since it increases the importance of a targeted node in the chain, e.g., to increase the popularity of a person in a social network, or increase the congestion in a road with little traffic. 
However, this leads to less interesting numerical results since we have proved that, in that case, the solution $D$, that is, the closed-form solution $\Delta(\al^*)$ described in Theorem~\ref{th:IminusGsolution}, leads to a sparse optimal rank-one solution as soon as $\lambda$ is sufficiently large; see Theorem~\ref{lem:optimalrankone}.   
\end{remark}
}


    


\subsubsection{Comparison of the different solutions}

We let $n = 1000$ 
and 
\revise{$k \in \{ 1,2,5,10,50,100, 500\}$}, and Table~\ref{tab:resultssyntn1000} 
reports 
the relative objective values (obj),  
 the average relative sparsity (spars) and the computational time for 10 randomly generated matrices.

\begin{table}[ht!]
    \centering
    \begin{tabular}{cc||c|c|c|c|c|c|c}
  &  & \revise{$MH$} &  $D$ & $S$ & $CG(10^{-2})$ &  $CG(10^{-4})$  &  $CG(0)$ & $GS$ \\ \hline     
 \multirow{3}{*}{\begin{sideways} $k=1$ \end{sideways} }      
 & obj (\%) & \revise{43.76} &  178.88 &  23.10 &  22.73 &  11.48 &  11.47 &  11.47 \\
 & spars (\%) & \revise{62.54} &  99.90 &  58.91 &  58.87 &  46.03 &  45.93 &  45.90 \\
 & time (s.) & \revise{0.01} &  0.20 &  0.03 &  0.06 &  0.99 &  1.57 &  25.53 \\  \hline   
   \multirow{3}{*}{\begin{sideways} $k=2$ \end{sideways} }   
& obj (\%) & \revise{58.71} &  175.53 &  11.62 &  11.21 &  7.51 &  7.51 &  7.51 \\
 & spars (\%) & \revise{59.93} &  99.90 &  31.92 &  31.62 &  25.29 &  25.29 &  25.29 \\
 & time (s.) & \revise{$<$ 0.01} &  0.19 &  0.04 &  0.08 &  0.63 &  0.79 &  25.16 \\ \hline   
    \multirow{3}{*}{\begin{sideways} $k=5$ \end{sideways} }   
 & obj (\%) & \revise{64.13} &  160.52 &  3.89 &  3.75 &  3.10 &  3.10 &  3.10 \\
 & spars (\%) & \revise{54.56} &  99.90 &  12.91 &  12.69 &  10.88 &  10.88 &  10.88 \\
& time (s.) & \revise{$<$ 0.01} &  0.04 &  0.08 &  0.19 &  0.74 &  1.30 &  25.03 \\  \hline   
     \multirow{3}{*}{\begin{sideways} $k=10$ \end{sideways} }   
& obj (\%) & \revise{65.23} &  105.36 &  1.87 &  1.82 &  1.69 &  1.69 &  1.69 \\
 & spars (\%) & \revise{52.39} &  99.90 &  6.21 &  6.14 &  5.50 &  5.50 &  5.50 \\
 & time (s.) & \revise{$<$ 0.01} &  0.02 &  0.22 &  0.48 &  1.57 &  1.97 &  25.61 \\   \hline   
      \multirow{3}{*}{\begin{sideways} $k=50$ \end{sideways} }   
& obj (\%) & \revise{66.78} &  19.72 &  1.02 &  0.94 &  0.94 &  0.94 &  0.94 \\
 & spars (\%) & \revise{50.51} &  99.90 &  1.52 &  1.41 &  1.40 &  1.40 &  1.40 \\
 & time (s.) & \revise{0.03} &  0.01 &  2.68 &  3.41 &  4.10 &  4.81 &  28.27 \\   \hline   
       \multirow{3}{*}{\begin{sideways} \hspace{-2mm} $k$\hspace{0.5mm}$=$\hspace{0.5mm}$100$  \end{sideways} }   
 & obj (\%) & \revise{66.57} &  13.65 &  1.40 &  1.39 &  1.39 &  1.39 &  1.39 \\
 & spars (\%) & \revise{50.26} &  99.89 &  1.21 &  1.20 &  1.20 &  1.20 &  1.20 \\
 & time (s.) & \revise{0.07} &  0.01 &  16.19 &  18.57 &  18.57 &  21.01 &  31.74 \\   \hline   
        \multirow{3}{*}{\begin{sideways} \hspace{-2mm} $k$\hspace{0.5mm}$=$\hspace{0.5mm}$500$  \end{sideways} }   
      & obj (\%) &   \revise{66.74} &  0.24 &  0.03 &  0.03 &  0.03 &  0.03 &  0.03 \\
 & spars (\%) &\revise{50.05} &  99.90 &  0.11 &  0.11 &  0.11 &  0.11 &  0.11 \\
& time (s.) & \revise{0.04} &  0.02 &  57.29 &  57.29 &  57.29 &  57.29 &  57.43 \\ 
    \end{tabular}
    \caption{Experiments on synthetic data sets reporting the average values for ten randomly generated  queue-like transition matrices (with about $2k$ non-zero entries per row) and $\hat \mub = G^\top \1_n/n$.} 
    \label{tab:resultssyntn1000}
\end{table}

We observe the following. 
\begin{itemize}
    \item In terms of relative objective values, $D$ performs rather poorly for small values of $k$. In particular, for $k \leq 10$, its relative error is larger than $100\%$ while the globally optimal solution, $GS$, is below $11\%$ in all cases. For $k=500$, when $G$ is dense, it actually provides a good solution, as predicted by theory because $\|\mub - \hat \mub\|_2$ is small~\eqref{eq:boundoptDelta} (the optimal $\Delta$ has a very small norm, 0.03\% of that of $G$).  
\revise{$MH$ does provide reasonable solutions for small $k$, although far from global optimality, and performs the worst for $k \geq 50$.} 
    The solution $S$ provides rather good solutions, especially when $k \geq 5$. This shows that restricting the solution to the support of $G+I$ allows us to recover reasonable solutions, compared to the global one.  
    $CG(10^{-2})$ performs similarly as $S$, while $CG(10^{-4})$ provides solutions of similar quality as $CG(0)$ which provides globally optimal solutions, as $GS$ does. 


    \item In terms of sparsity, the support of $D$ coincide with that of $G+I$ except for one row, as expected (see Lemma~\ref{lem:sparsityDal}),  explaining the 99.9\% relative sparsity in all cases (since $n= 1000$).
\revise{the sparsity of $MH$ is around 50\% of that of $G+I$ because the M-H algorithm does not use the full support (some rows of $\Delta$ are zeros).} 
    All the other solutions are significantly sparser than $G+I$, especially when $k$ grows. This shows that minimizing the component-wise $\ell_1$ norm is able to promote sparsity; see Theorem~\ref{lem:sparsitysol}. For example, when $G$ is dense ($k=500$), the optimal solution only has 0.11\% of non-zero entries.  Even for $k=1$, where $G+I$ has 3 non-zero per rows (except the first and last row which have only 2), the optimal solution has less than 50\% of these zeros, meaning 1.5 non-zeros on average per row for the optimal solution (which could potentially use all the $n^2$ entries of $\Delta$). 

    \item In terms of computational time, computing \revise{$MH$ and} $D$ is extremely fast, as expected (the reason why the case $k=1,2$ is slightly slower for $D$ is explained in Remark~\ref{rem:cond}). 
    Computing $S$ is fast for sparse $G$, and becomes more expensive as $k$ grows; see section~\ref{sec:salabsynt} for an extensive numerical experiment on this case. 
    The computational cost of $CG(\delta)$ grows as $\delta$ decreases, as expected, but remains close to that of $S$ (remember, $S$ is computed at the first step of Algorithm~\ref{alg:CGtsdp}, that is, $S = CG(\delta)$ for any $\delta$ sufficiently large). The reason is that the computational cost of each iteration of the column-generation approach (Algorithm~\ref{alg:CGtsdp}) is proportional to the computation of $S$: the number of columns added is roughly the same at each step. 
Finally, $CG(0)$ is significantly faster than $GS$, as expected, since $GS$ does not leverage the column generation and solves the full problem at once, which requires forming a large constraint matrix. 
Note that, when $G$ is dense ($k=500$), $S$, $CG(\delta)$ and $GS$ coincide since they all solve the full problem, on $\supp(G+I)$, which is time consuming. 
    
\end{itemize}

\begin{remark}[Conditioning] \label{rem:cond}
To compute $\Delta(\al^*)$, the stationary distribution $\mub$ of $G$ is needed. 
    For large $n$ and small $k$, the condition number to calculate $\mub$ is very high; see~\cite{kirkland2010column} for a discussion on this issue. 
    For example, for $n=1000$ and $k=1$, the \texttt{eigs} function of MATLAB runs into numerical problems and cannot return a feasible solution (it return NaN for the largest eigenvalue, and $\mub$ with negative entries).  
     In this case, we resort to 100 iterations of the power method, initialized with $\mub = \1_n/n$, for simplicity. 
\end{remark}

\subsubsection{Scalability to compute $S$} \label{sec:salabsynt}

Empirically, we have observed that the number of iterations of Algorithm~\ref{alg:CGtsdp} is relatively small, especially when $\delta$ is large; see Figure~\ref{fig:morenoCG} for an experiment on real data sets. 
Moreover, we have observed that the computational cost of each iteration is similar, because we add to the support, $\Omega$, roughly the same number of \revise{index pairs}. Hence, the  cost of computing $CG(\delta)$ is the cost of computing $S$ times the number of iterations, which is typically small; for example, for $\delta = 10^{-2}$, it never exceeded 3 iterations. Hence it makes sense to analyse the computational cost of computing $S$, to get an idea how Algorithm~\ref{alg:CGtsdp} scales with $n$ 
and~$k$. Figure~\ref{fig:timeSall} reports  the total time to compute $S$ for various values of $k \in \{1, 2, 4, 8, 12, 16, 32, 64, 128\}$ 
and
\[
n \in \{1000, 1669, 2783, 4642, 7743, 12916, 21545, 35939, 59949, 10^5, \revise{2 \cdot 10^5}\} 
\] 
with log-spaced values\footnote{We used \texttt{[ceil(logspace(3, 5, 10)), 2*1e5]} in MATLAB.}.  
To limit the computational time, we combine values of $n$ and $k$ only when $ \nnz(G) \approx 2nk \leq 450,000$. Above 450,000, the computational time exceeds 200 seconds. 
\begin{figure}[ht!]
\begin{center}
\begin{tabular}{c}
\includegraphics[width=0.8\textwidth]{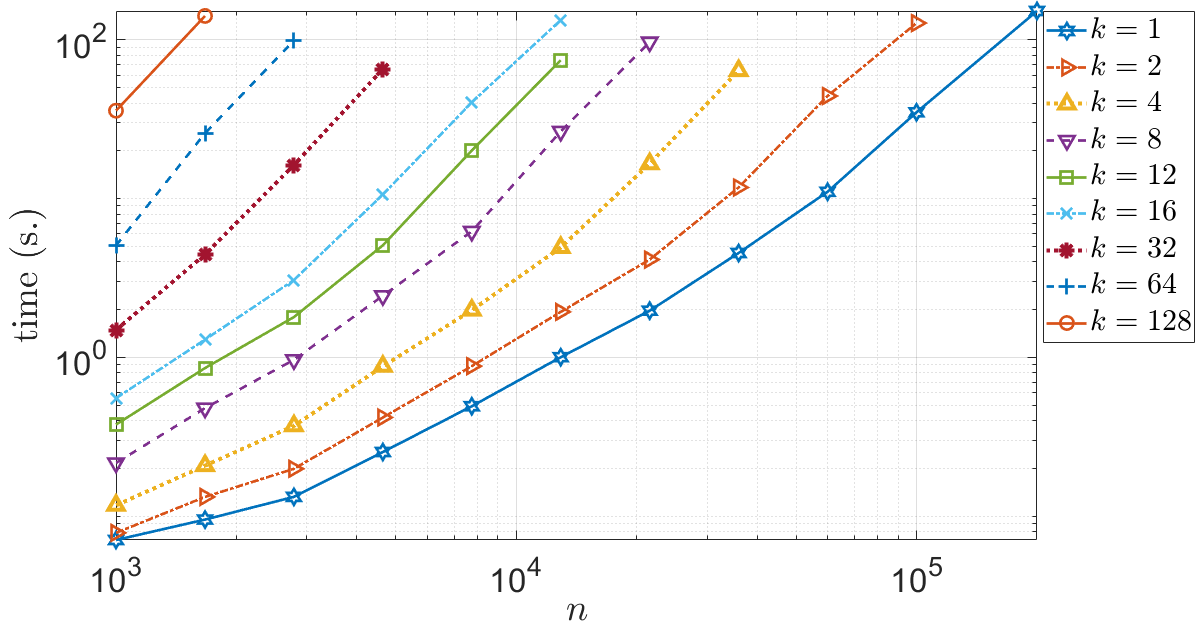} 
\end{tabular}
\caption{Average computational time to compute the solution $S$, that is, the solution of~\eqref{eq:sparseTSDP} with $\Omega = \supp(G+I)$, for five queue-like transition matrices generated randomly (with about $2k$ non-zero entries per row) and $\hat \mub = G^\top \1_n/n$. This time accounts for the formulation  of~\eqref{eq:sparseTSDP} and its resolution.}  
\label{fig:timeSall}
\end{center}
\end{figure}

There is a dependence between the computational time and $n$ that is somewhat linear, in a log-log plot. 
The average slope for computational times larger than 1 second is 2.2, meaning that the time depends roughly  quadratically on $n$, for $k$ fixed.  Similarly, the average slope in a log-log plot between the computational time and $k$ is 1.98, hence the time also depends quadratically on $k$, for $n$ fixed. 
This allows us to solve large problems relatively fast, e.g., \revise{$n = 2 \cdot 10^5$} and $k=1$ in about 150 seconds. It turns out that, quite surprisingly, the bottleneck of our algorithm is not to solve the large (sparse) linear optimization problems, but to form the constraint matrix $A$, of size\footnote{Recall the number of variables in the linear programs 
are $2\nnz(G)$ for $\Delta^+$ and $\Delta^-$, and $n$ for $\Delta^0$ since the diagonal entries of our synthetic data sets $G$ are zeros.} $2n \times (2\nnz(G)+n)$, with two non-zeros per column of $A$. 
Constructing such a large sparse matrix does not require a linear time in the number of non-zeros (there are about $4 \nnz(G)$) but larger due to the access to memory, among other things. 

Figure~\ref{fig:timeSlinopt} reports the time to solve the linear optimization problems to obtain $S$. Comparing to Figure~\ref{fig:timeSall}, we see that the time to solve the linear optimization problems is an order of magnitude smaller than formulating the problem. 
\begin{figure}[ht!]
\begin{center}
\begin{tabular}{c}
\includegraphics[width=0.8\textwidth]{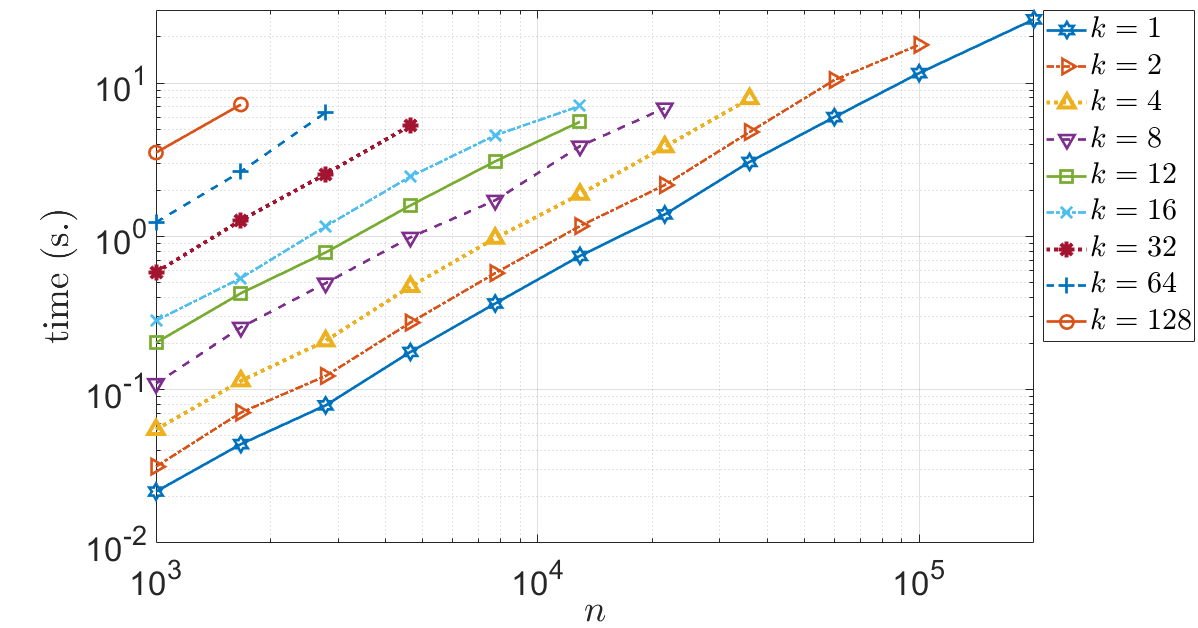} 
\end{tabular}
\caption{Average computational time to solve  of~\eqref{eq:sparseTSDP} with $\Omega = \supp(G+I)$, for 5 queue-like transition matrix generated randomly (with about $2k$ non-zero entries per row) and $\hat \mub = \1_n/n$. This time accounts only for the resolution~\eqref{eq:sparseTSDP}, not the formulation.}  
\label{fig:timeSlinopt}
\end{center}
\end{figure} 
Possibly the construction of the sparse constraint matrix could be accelerated by using another language than MATLAB, e.g., C or C++.

In summary, the proposed approach to solve the $\ell_1$-TDSP with support constraint requires $O(\nnz(G))$ memory and, empirically, about $O(\nnz(G)^2)$ time.   
Figure~\ref{fig:nnzGvstime} shows the time required to solve~\eqref{eq:sparseTSDP} with $\Omega = \supp(G+I)$ as a function of $\nnz(G) \approx 2nk$ for the synthetic data sets (these are the same values as in Figure~\ref{fig:timeSall} but presented differently), which follows a quadratic trend, as noted above, namely, $\text{time} \approx \zeta \nnz(G)^2$ with \revise{$\zeta = 9 \cdot 10^{-10}$}.  
\begin{figure}[ht!]
\begin{center}
\begin{tabular}{c}
\includegraphics[width=0.5\textwidth]{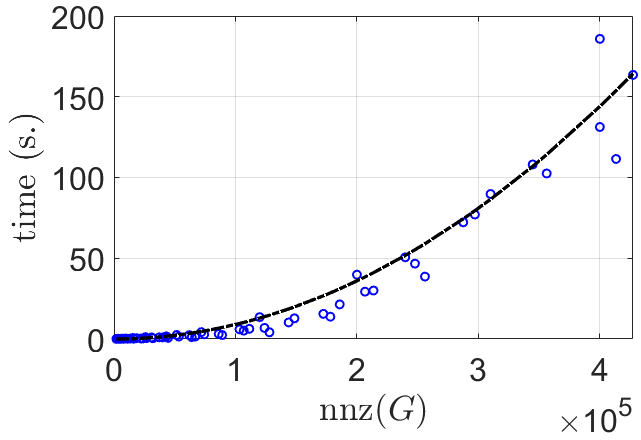} 
\end{tabular}
\caption{Average computational time to formulate and solve~\eqref{eq:sparseTSDP} with $\Omega = \supp(G+I)$ for 5 queue-like transition matrix generated randomly (with about $2k$ non-zero entries per row) and $\hat \mub = \1_n/n$. The black curve is $\zeta \nnz(G)^2$ with \revise{$\zeta = 9 \cdot 10^{-10}$}. 
These are the same values as in Figure~\ref{fig:timeSall}.}  
\label{fig:nnzGvstime}
\end{center}
\end{figure}

\subsection{Real data sets} \label{sec:realdata}

In this section, we use the same three sparse data sets as in~\cite{BerkhoutHV23}; see Table~\ref{tab:my_data}. 
\revise{They come from the KONECT project 
(\url{http://konect.cc/}).}  
\begin{table}[ht!]
    \centering
    \begin{tabular}{c|ccc}
data sets   & email   & road & moreno \\ \hline       
size ($n$)      & 1133 & 1030   & 2155  \\
   average number of \revise{non-zeros} per row    
   & 9.62   & 2.51  & 5.32  \\ 
    \end{tabular}
    \caption{Sparse real data sets from~\cite{BerkhoutHV23}.} 
   \label{tab:my_data}
\end{table}    
They represent the following: 
\begin{itemize}
    \item email: email-conversation network of university employees at the University of Rovira i Virgili\footnote{\revise{\url{http://konect.cc/networks/arenas-email/}}}. \revise{Nodes are employees
and each undirected edge represents that at least
one email was sent between the employees. The support of $G$ is symmetric, but $G$ is not symmetric after the normalization required to make $G$ stochastic.}   

    \item road: road network between the largest cities in Europe\footnote{\revise{\url{http://konect.cc/networks/subelj_euroroad/}}}. \revise{As for the previous data set, the support of $G$ is symmetric, but $G$ is not symmetric.}   

    \item moreno: high-school network of student relationships from a survey from
1994/1995 on a high school where each student had to indicate his/her 5 best female and male friends\footnote{\revise{\url{http://konect.cc/networks/moreno_highschool/}}}. \revise{This is a directed network, that is, the support of $G$ and $G$ are not symmetric.}  
\end{itemize}

We use $\hat{\mu} = (1- \revise{\epsilon}) \mu + \epsilon \frac{ \mathbf 1_n}{n}$ for $\epsilon \in \{0.01, 0.1, 0.5\}$.  \revise{This means that we are trying to balance the importance of the nodes in the chain, e.g., for the road network, to make the roads more evenly congested; see Remark~\ref{rem:choicemuh}.} 
   Table~\ref{tab:resultsreal001} reports the results for $\epsilon = 0.01$, 
   Table~\ref{tab:resultsreal01} for $\epsilon = 0.1$, and 
   Table~\ref{tab:resultsreal05}  for $\epsilon = 0.5$. 
\begin{table}[ht!]
    \centering
    \begin{tabular}{cc||c|c|c|c|c|c|c}
  & &  \revise{$MH$} & $D$ & $S$ & $CG(10^{-2})$ &  $CG(10^{-4})$  &  $CG(0)$ & $GS$ \\ \hline     
 \multirow{3}{*}{\begin{sideways} email \end{sideways} } 
  & obj (\%) & \revise{3.82} & 5.02 &  0.42 &  0.20 &  0.17 &  0.17 &  0.17 \\ 
 & spars (\%) & \revise{53.44} & 99.40 &  11.77 &  10.64 &  10.26 &  10.25 &  10.25 \\ 
 & time (s.) & \revise{0.03} & 0.03 &  0.15 &  0.32 &  0.86 &  1.15 &  36.31  \\ \hline  
 \multirow{3}{*}{\begin{sideways} road \end{sideways} } 
 & obj (\%) & \revise{0.72} & 1.92 &  0.41 &  0.36 &  0.31 &  0.28 &   0.28 \\ 
 & spars (\%) & \revise{45.96} & 99.70 &  35.74 &  35.96 &  33.57 &  31.52 &  31.16  \\ 
 & time (s.) & \revise{0.02} & 0.02 &  0.03 &  0.08 &  0.25 &  1.26 &  27.25 \\ \hline   
 \multirow{3}{*}{\begin{sideways} moreno \end{sideways} } 
   & obj (\%) & \revise{116.54}  & 21.73 &  4.90 &  0.98 &  0.79 &  0.78 &  0.78 \\ 
  & spars  (\%) & \revise{81.67} & 99.93 &  24.42 &  21.72 &  20.14 &  20.03 &  20.03   \\ 
  & time (s.) & \revise{0.11} &  0.09 &  0.21 &  0.94 &  3.98 &  6.41 &  257.86  \\   
    \end{tabular}
    \caption{Experiments on real data with $\epsilon = 0.01$.} 
    \label{tab:resultsreal001}
\end{table} 

\begin{table}[ht!]
    \centering
    \begin{tabular}{cc||c|c|c|c|c|c|c}
  & &\revise{$MH$}&  $D$ & $S$ & $CG(10^{-2})$ &  $CG(10^{-4})$  &  $CG(0)$ & $GS$ \\ \hline    
 \multirow{3}{*}{\begin{sideways} email \end{sideways} } 
 & obj (\%) & \revise{26.68} &   36.92 &  4.23 &  2.13 &  2.13 &  2.13 &  2.13 \\ 
 & spars (\%) & \revise{53.44} & 99.40 &  12.94 &  12.84 &  12.83 &  12.84 &  12.83  \\ 
 & time (s.) & \revise{0.03} & 0.03 &  0.12 &  0.48 &  0.68 &  0.89 &  37.05  \\ \hline   
 \multirow{3}{*}{\begin{sideways} road \end{sideways} } 
 & obj (\%) & \revise{6.67}  &  18.39 &  4.01 &  3.60 &  2.86 &  2.86 &  2.86 \\ 
 & spars (\%) &  \revise{45.96} & 99.70 &  35.96 &  35.96 &  31.49 &  31.43 &  31.16 \\ 
 & time (s.) & \revise{0.02} & 0.02 &  0.03 &  0.08 &  0.99 &  1.85 &  27.33 \\ \hline        
 \multirow{3}{*}{\begin{sideways} moreno \end{sideways} } 
  & obj (\%) & \revise{116.54} & 69.92 &  18.29 &  6.40 &  5.61 &  5.60 &  5.60 \\ 
 & spars (\%) & \revise{81.67} & 99.93 &  29.88 &  22.94 &  21.52 &  21.51 &   21.51 \\ 
 & time (s.) & \revise{0.11}  & 0.09 &  0.21 &  1.43 &  5.73 &  7.20 &  255.08  \\ 
    \end{tabular}
    \caption{Experiments on real data with $\epsilon = 0.1$.} 
    \label{tab:resultsreal01}
\end{table} 

 \begin{table}[ht!]
    \centering
    \begin{tabular}{cc||c|c|c|c|c|c|c}
  & &   \revise{$MH$} & $D$  & $S$ & $CG(10^{-2})$ &  $CG(10^{-4})$  &  $CG(0)$ & $GS$ \\ \hline      
 \multirow{3}{*}{\begin{sideways} email \end{sideways} } 
 & obj (\%) & \revise{67.65} &  107.47 &  28.22 &  24.93 &  24.93 &  24.93 &   24.93 \\ 
 & spars (\%) & \revise{53.44} & 99.40 &  33.64 &  35.50 &  35.52 &  35.52 &  35.51 \\ 
& time (s.) &  \revise{0.03}  & 0.03 &  0.25 &  0.88 &  1.16 &  1.16 &  38.22  \\ \hline   
 \multirow{3}{*}{\begin{sideways} road \end{sideways} } 
 & obj (\%) &  \revise{26.57} &  80.85 &  19.73 &  18.78 &  16.82 &  16.82 &  16.82  \\ 
 & spars (\%) & \revise{45.96} & 99.70 &  40.70 &  40.45 &  38.45 &  38.23 &   37.96  \\ 
& time (s.) &  \revise{0.02} & 0.02 &  0.04 &  0.08 &  0.75 &  1.18 &  27.31 \\    \hline     
 \multirow{3}{*}{\begin{sideways} moreno \end{sideways} } 
  & obj (\%) &  \revise{117.75}  & 135.08 &  39.00 &  21.84 &  21.45 &  21.45 &  21.45  \\ 
  & spars (\%) & \revise{81.67} & 99.93 &  40.08 &  29.97 &  29.34 &  29.34 &  29.33   \\ 
 & time (s.) & \revise{0.11} & 0.10 &  0.24 &  2.32 &  5.51 &  5.51 &  253.31  \\ 
    \end{tabular}
    \caption{Experiments on real data with $\epsilon = 0.5$.} 
    \label{tab:resultsreal05}
\end{table}


We observe the following: 
\begin{itemize}

\item In terms of objective function values, $D$ is significantly worse than $GS$, especially when $\epsilon$ is large: the objective function values of $D$ gets  smaller as $\epsilon$ gets smaller, that is, as  $\|\mub - \hat \mub\|_2$ gets smaller, as expected; see~\eqref{eq:boundoptDelta}.  
\revise{$MH$ provides good solutions for the road data set. An explanation is that the support of $G$ is symmetric, as it is a road network. For the moreno data set, whose support is asymmetric (43.55\% of non-zeros of $G$ are not in $G^\top$), it performs poorly. For the email data set, whose support is symmetric as well, it performs better than $D$ but is far from optimality.} 
$S$ provides a solution with objective relatively close to that of $GS$, except for the moreno data set where it is significantly worse (e.g., for $\epsilon = 0.1$, obj of $S$ is 18.29\% while it is 5.60\% for $GS$).  
All $CG(\delta)$ variants perform similarly with slight improvements when $\delta$ decreases. 
As expected, the objectives of $CG(0)$ and $GS$ coincide. 

\item In terms of computational cost, \revise{$MH$ and $D$ are again extremely fast.}  
$CG(0)$ outperforms $GS$, as for synthetic data sets since $G$ is sparse. 
However, $CG(0)$ sometimes requires more time than $CG(10^{-2})$ and $CG(10^{-4})$ for a negligible improvement in the objective (e.g., for the moreno data set with $\epsilon = 0.01$, from 3.98 seconds to 6.41 seconds to reduce the relative objective from 0.79\% to 0.78\%). This is because $CG(0)$ can only stop when it has found an optimal solution.  
The solution $S$ is computed extremely fast (less than 0.25 seconds in all cases) but sometimes produces high objective function values. 
Hence, in practice, using $CG$ with a value of $\delta \in [10^{-2}, 10^{-4}]$  
seems to be a good compromise between the quality of the solution and the run time. Figure~\ref{fig:morenoCG} shows the evolution of the objective function values between the iterations of Algorithm~\ref{alg:CGtsdp} (that computes $CG$). 
\begin{figure}[ht!]
\begin{center}
\begin{tabular}{c}
\includegraphics[width=0.7\textwidth]{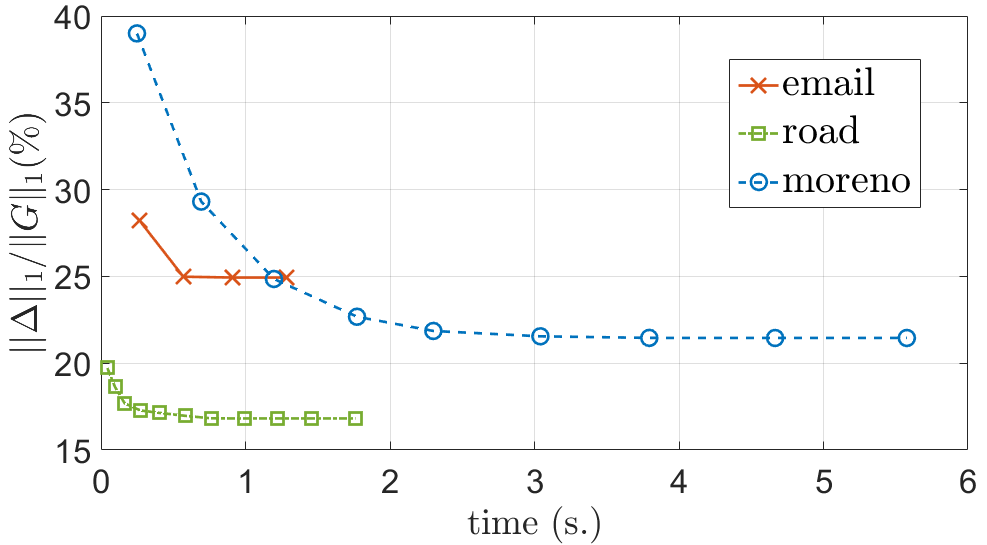} 
\end{tabular}
\caption{Evolution of the relative objective function between each iteration of the column-generation approach (Algorithm~\ref{alg:CGtsdp}) for the three real data sets, for $\epsilon = 0.5$. (Note that the timings are not exactly the same as in Table~\ref{tab:resultsreal05}, this is  because we regenerated the solutions to make this plot.)} 
\label{fig:morenoCG}
\end{center}
\end{figure}
This confirms that a few iterations of Algorithm~\ref{alg:CGtsdp} provides good solutions.  



    \item In terms of sparsity, $D$ has essentially the same sparsity of $G+I$, except for one row; this is expected, see Lemma~\ref{lem:sparsityDal}.  
    \revise{$MH$ has a sparsity close to 50\% for the symmetric data sets, as for the synthetic data sets, and around 80\% for the asymmetric one, moreno.} 
    The solution $S$ is significantly sparser than $G+I$, between 3 and 10 times. 
    The global solution ($GS$ and $CG(0)$) generates surprisingly sparse solutions, slightly sparser than $S$.  
     In quite a few cases, the globally optimal solutions generated by $GS$ and $CG(0)$ do not  have the same sparsity, meaning that the solution of~\eqref{eq:sparseTSDP} is not unique. 



\end{itemize}




\section{Concluding remarks}  \label{sec:conclude}

In this paper we proposed several algorithms for assigning a target stationary distribution $\hat\mub$ to a perturbed 
stochastic matrix $\hat G=G+\Delta$, with a constraint on the support of $\Delta$. 
We first analyzed the special case where $\Delta := \diag(\al)(I_n-G)$ whose support is restricted to the union of the supports of $G$  and the identity, which implies that it does not destroy the sparsity of the original  matrix $G$. We proved several properties of this solution: its optimality for that class of perturbations, 
its sparsity, when it is of rank-one, and when its norm is minimized depending the ordering of $\hat \mub$ compared to $\mub$. 
Unfortunately, numerical experiments show that this proposed solution is in most cases quite far
from global optimality, because the feasible set of perturbations is too constrained, while not being very sparse as its support essentially coincides with that of $G+I$. \revise{It works well however when $\| \mub-\hat \mub \|$ is small enough (see Equation~\eqref{eq:boundDal}, and Table~\ref{tab:resultssyntn1000} when $k=500$), 
and is optimal when $\hat \mub = \frac{1}{1+\lambda}(\mub + \lambda \e_j)$  
for any $j$ and $\lambda$ that is sufficiently large (Theorem~\ref{lem:optimalrankone}).} 
We then proposed an effective linear optimization formulation of the problem when minimizing the component-wise $\ell_1$ norm of $\Delta$ which promotes sparse solutions, as proved in Theorem~\ref{lem:sparsitysol}. To solve this linear optimization problem efficiently, we designed a dedicated column generation approach; see Algorithm~\ref{alg:CGtsdp} which can be stopped before global \revise{optimality} to improve the trade off between solution quality and run time. Algorithm~\ref{alg:CGtsdp} allows us to solve large sparse problems, with and without support constraint and up to global optimality, extremely fast, for sparse matrices up to size $n = 200,000$ in a few minutes. This is because empirically the main computational cost of the method is to construct a large sparse matrix of dimension $2n \times O(\nnz(G))$ with two non-zeros per column.  

A limitation of our column generation approach is that it relies on the support of $G+I$ to provide a first feasible solution. If $G$ is dense, this will make the algorithm not scale as well. 
However, even for dense matrices, the optimal solutions are (often) sparse; see, e.g., the last rows of Table~\ref{tab:resultssyntn1000}. 
Hence it would be interesting to address the following problem: Given $G$ and $\hat \mub$, provide a sparse set $\Omega$ so that the TSDP with support constraints~\eqref{eq:sparseTSDP} is feasible.  
This would not only allow us to initialize Algorithm~\ref{alg:CGtsdp} more efficiently, but also provide sparse solutions to the TSDP. It could be useful even when $G$ is sparse by initializing $\Omega$ with a support sparser than $G+I$, leading to computational gains. 
We conjecture that $\Omega$ needs, generically, at least $2n$ elements~: one per column to satisfy $\hat \mub^\top  \Delta = \hat \mub^\top (I-G)$, and 2 per row to satisfy $\Delta \1_n = 0$.


\revise{
\section*{Acknowledgment} 

We thank Sam Power (University of Bristol) for letting us know of the use of the Metropolis-Hastings algorithm in the context of the TSDP, and for helpful discussions. 

We are also very grateful to the two anonymous reviewers who carefully read the manuscript; their feedback allowed us to improve it significantly. 
}

\appendix

\section{Proof of Theorem~\ref{th:ordering}}  \label{sec:appendixA} 

\begin{lemma} \label{lem:ab}
Let $0<a_- \le a_+$ and $0<b_-\le b_+$, then
$$ \frac{a_-}{b_+} \le \min\left( \frac{a_-}{b_-}, \frac{a_+}{b_+}\right) \le \max\left( \frac{a_-}{b_-}, \frac{a_+}{b_+} \right)  \le  
\frac{a_+}{b_-}, 
$$
which implies that $ \left[\min( \frac{a_-}{b_-}, \frac{a_+}{b_+}) , \max( \frac{a_-}{b_-},\frac{a_+}{b_+} )\right]
\subset  \left[  \frac{a_-}{b_+}, \frac{a_+}{b_-}\right]$. \\
This inclusion is strict if and only if $a_-<a_+$ and $b_-<b_+$.
\end{lemma}
\begin{proof}
It follows from the assumptions that 
$$  a_-b_- \le a_-b_+, a_+b_- \le a_+b_+.
$$
Dividing all quantities by the positive product $b_-b_+$ yields
$$  \frac{a_-}{b_+} \le  \frac{a_-}{b_-}, \frac{a_+}{b_+} \le  \frac{a_+}{b_-},
$$
from which the inclusion result follows. It is easy to see that if either $a_-=a_+$ or $b_-=b_+$, then both intervals are equal, and that if $a_-<a_+$ and $b_-<b_+$, then the inclusion of the intervals is strict.
\end{proof}

\begin{lemma} \label{lem:decrease}
Let the elements of the vectors $\mub$, $\hat\mub$ and $\r:=\mub./\hat\mub$ be positive and assume those of $\hat \mub$ are ordered in a non-decreasing way, that is, 
$0 < \hat\mu_1 \le \hat\mu_2 \le \ldots \hat\mu_n$. 
Let the elements of the permuted vector $\tilde \mub:=\tilde P \mub$ also be ordered in a non-decreasing way, that is, $0 < \tilde\mu_1 \le \tilde\mu_2 \le \ldots \tilde\mu_n$. 
Then 
\begin{equation} \label{twobytwo}
\revise{\left[ \min_i \tilde \mu_i/\hat \mu_i , \max_i \tilde \mu_i/\hat \mu_i  \right] \subset
  \left[ \min_i \mu_i/\hat \mu_i , \max_i \mu_i/\hat \mu_i  \right] .}
\end{equation} 
\end{lemma}
\begin{proof}
The permutation $\tilde P$ that reorders the elements of the vector $\mub$ to $\tilde \mub := \tilde P\mub$, can be factored in a sequence of reorderings of just two elements that are not ordered in a non-decreasing way. At each such step, the interval \eqref{twobytwo} in which the new ratios lie, can only decrease because of Lemma \ref{lem:ab}.
\end{proof}

\revise{

\section{Rank-one solution for the component-wise $\ell_1$  norm}  \label{app:proofL1}

Let us show that there exists an optimal rank-one  solution to the TSDP for the component-wise $\ell_1$ norm without the nonnegativity constraints. 
This can be done by proving that there exists a rank-one optimal solution to the problem 
\begin{equation} \label{eq:l1sol}
\min_{\Delta} \|\Delta\|_1 
\quad 
\text{ such that } 
\quad 
\hat \mub^\top \Delta = \hat  \mub^\top (I-G), 
\end{equation}
which satisfies the sum-to-one condition, that is, $(G+\Delta) \1_n = \1_n$. 

Let us denote $\mathbf b :=  (I-G)^\top \hat \mub$. The problem~\eqref{eq:l1sol} is separable by columns of $\Delta$, that is, the columns of $\Delta$ can be optimized independently. The subproblem of~\eqref{eq:l1sol} for the $k$th column of $\Delta$ is 
\begin{equation} \label{eq:l1solsub}
\min_{\mathbf z} \| \mathbf z \|_1 
\quad 
\text{ such that } 
\quad 
 \hat \mub^\top \mathbf z = b_k. 
\end{equation} 
We need to find the $\mathbf z$ with smallest $\ell_1$ norm such that $\hat \mub^\top \mathbf z = b_k$. It is easy to convince oneself that $\mathbf z$ with a single non-zero entry corresponding to the largest entry in $\hat \mub$ is optimal. 
An optimal solution of \eqref{eq:l1solsub} is therefore given by 
$\mathbf z^* = \frac{b_k}{\hat \mu_i} \e_i$ for $i \in \argmax_k \hat \mub_k$. 
Therefore, 
\[
\Delta^* 
\; = \; 
\frac{1}{\hat \mu_i} \ \e_i \ \mathbf b^\top  
\; = \; 
\frac{1}{\hat \mu_i} \ \e_i \ \hat \mub^\top (I-G), 
\] 
is a rank-one optimal solution to~\eqref{eq:l1sol}. 
Moreover, $G+\Delta^*$ respects the sum-to-one condition, as $\Delta^* \1_n = 0$, which concludes the proof. 
}


\bibliographystyle{spmpsci}
\bibliography{biblio}

\end{document}